\documentclass[10pt]{amsart}
\usepackage{graphicx,amssymb,amsmath,amsfonts,amscd,hyperref,youngtab}
\newtheorem{thm}{Theorem}[section]
\newtheorem{cor}[thm]{Corollary}
\newtheorem{lem}[thm]{Lemma}
\newtheorem{prop}[thm]{Proposition}
\newtheorem{defn}[thm]{Definition}
\newtheorem{rem}[thm]{Remark}
\newtheorem{exa}[thm]{Example}

\newcommand{\Trace}{\mathrm{Trace}}
\newcommand{\Nm}{\mathrm{N}}

\newcommand{\Sym}{\mathrm{Sym}}
\newcommand{\Hom}{\mathrm{Hom}}
\newcommand{\Fq}{\mathbb F_q}

\newcommand{\CC}{\mathcal C}
\newcommand{\FFq}{{\bar{\mathbb F}_q}}
\newcommand{\QQ}{\bar{\mathbb Q}_\ell}
\newcommand{\R}{\mathrm R}

\newcommand{\Gal}{\mathrm {Gal}}
\newcommand{\AAA}{\mathbb A}
\newcommand{\PP}{\mathbb P}
\newcommand{\FF}{\mathcal F}
\newcommand{\GG}{\mathbb G}
\newcommand{\GGG}{\mathcal G}

\newcommand{\HH}{\mathrm H}
\newcommand{\HHH}{{\mathcal H}}
\newcommand{\LL}{{\mathcal L}}

\newcommand{\Tr}{\mathrm {Tr}}

\newcommand{\Spec}{\mathrm{Spec\:}}
\newcommand{\Dbc}{{\mathcal D}^b_c}

\newcommand{\Frob}{\mathrm{Frob}}
\newcommand{\Swan}{\mathrm{Swan}}
\newcommand{\Sha}{{\mathcal Sh}(X,\QQ)}
\newcommand{\Shg}{{\mathcal Sh}(\GG_{m,k},\QQ)}
\newcommand{\Sh}{\mathcal Sh}
\newcommand{\Perv}{{\mathcal P}erv}
\newcommand{\PPP}{\mathcal P}
\newcommand{\GL}{\mathrm{GL}}
\newcommand{\FT}{\mathrm{FT}}

\title{Rationality of Trace and Norm L-functions}
\author{Antonio Rojas-Le\'on}
\address{Departamanto de \'Algebra,
Universidad de Sevilla, Apdo 1160, 41080 Sevilla, Spain}
\address{E-mail: arojas@us.es}
\thanks{Partially supported by P08-FQM-03894 (Junta de Andaluc\'{\i}a), MTM2007-66929 and FEDER}

\begin{document}

\renewcommand{\thefootnote}{}

\begin{abstract}
For a given $\ell$-adic sheaf $\FF$ on a commutative algebraic group $G$ over a finite field $k$ and an integer $r\geq 1$ we define the $r$-th local norm $L$-function of $\FF$ at a point $t\in G(k)$ and prove its rationality. This function gives information on the sum of the local Frobenius traces of $\FF$ over the points of $G(k_r)$ (where $k_r$ is the extension of degree $r$ of $k$) with norm $t$. For $G$ the one-dimensional affine line or the torus, these sums can in turn be used to estimate the number of rational points on curves or the absolute value of exponential sums which are invariant under a large group of translations or homotheties.
\end{abstract}

\maketitle

\section{Introduction}

Let $k=\Fq$ be a finite field of characteristic $p$ and $V\subseteq\AAA^{n}_k$ a geometrically irreducible affine variety of dimension $r$. To fix ideas, suppose that its $\ell$-adic cohomology groups $\HH^i_c(V\otimes\bar k,\QQ)$ vanish for $i\neq r,2r$ (e.g. $V$ a smooth hypersurface with smooth section at infinity). If $\dim\HH^r_c(V\otimes\bar k,\QQ)=d$, the famous Deligne-Weil bound \cite[Th\'eor\`eme 8.1]{deligne1974conjecture} gives an estimate
$$
\left|\#V(k)-q^r\right|\leq d\cdot q^\frac{r}{2}
$$
for the number of rational points of $V$. Suppose now that we have a large finite abelian group $G$ acting on $V$. Then $G$ has an induced Frobenius invariant action on $\HH^r_c(V\otimes\bar k,\QQ)$, so this vector space splits as a direct sum $\bigoplus_\psi\HH^r_c(V\otimes\bar k,\QQ)^\psi$, where the sum is indexed by the set of characters of $G$ and $\HH^r_c(V\otimes\bar k,\QQ)^\psi$ is the subspace of $\HH^r_c(V\otimes\bar k,\QQ)$ on which $G$ acts via $\psi$. In particular, we get a decomposition
$$
\#V(k)-q^r=(-1)^r\Tr(\Frob_k|\HH^r_c(V\otimes\bar k,\QQ))=(-1)^r\sum_\psi\Tr(\Frob_k|\HH^r_c(V\otimes\bar k,\QQ)^\psi)
$$

If the situation is generic enough, one would expect that there should be some cancellation among the terms of the sum, thus giving a significative improvement of the Deligne-Weil bound if $\#G$ (and thus the number of terms in the sum) is large.
For instance, if $k'={\mathbb F}_{q'}$ is a subfield of $k$, the Artin-Schreier curve $y^{q'}-y=f(x)$ for $f\in k[x]$ has a natural action of the additive group $k'$ (where $t$ acts by $(x,y)\mapsto (x,y+t)$. In \cite{rlwan2010} this fact was used to give an improvement of the Weil bound for the curve of order of magnitude $\sqrt{q'}$.

In the same vein, suppose that $\FF$ is an $\ell$-adic sheaf on $\AAA^1_k$ which is invariant under translation by elements of $k'$. Equivalently, $\FF$ is the pull-back by the \'etale map $x\mapsto x^{q'}-x$ of a sheaf $\GGG$ on $\AAA^1_k$. Then 
$$
\sum_{t\in k}\Tr(\Frob_{k,t}|\FF_{\bar t})=\sum_{t\in k}\Tr(\Frob_{k,t^{q'}-t}|\GGG_{\bar t^{q'}-\bar t})=q'\cdot\sum_{\Tr_{k/k'}(u)=0}\Tr(\Frob_{k,u}|\GGG_{\bar u})
$$
Similarly, if $\FF$ on $\GG_{m,k}$ is invariant under the group of homotheties with ratios in $k'^\star$ (so $\FF$ is the pull-back of a sheaf $\GGG$ on $\GG_{m,k}$ under the $(q'-1)$-th power map), we can write
$$
\sum_{t\in k^\star}\Tr(\Frob_{k,t}|\FF_{\bar t})=(q'-1)\cdot\sum_{\Nm_{k/k'}(u)=1}\Tr(\Frob_{k,u}|\GGG_{\bar u})
$$
If $\FF$ is pure of weight $0$ and has no geometrically constant components, the Weil bound for the sum has order $O(d\sqrt q)$, where $d=\dim\HH^1_c(\AAA^1_{\bar k},\FF)$ (resp. $\dim\HH^1_c(\GG_{m,\bar k},\FF)$). On the other hand, for the right hand side we expect an estimate of the form $O(e\sqrt{q'}^{[k:k']+1})$ where $e=\dim\HH^1_c(\AAA^1_{\bar k},\GGG)$. Now since $\FF$ is the pull-back of $\GGG$ under a map of degree $\sim q'$, for general $\FF$ we should have $d\sim q'e$. This gives an estimate $O(d\sqrt{q'}^{[k:k']-1})$ for the second sum, which improves the Weil estimate by a factor of $\sqrt{q'}$. 

For instance, if $\psi:k\to\QQ^\star$ (respectively $\chi:k^\star\to\QQ^\star$) is an additive (resp. multiplicative) character of $k$, and $f\in k[x]$ is a polynomial of the form $g(x^{q'}-x)$ or of the form $g(x^{q'-1})$ with $g\in k[x]$ of degree $d$, the classical Weil bound for the exponential sum $\sum_{x\in k}\psi(f(x))$ (resp. $\sum_{x\in k}\chi(f(x))$) is $\cong (dq'-1)\sqrt{q}\cong dq'^{\frac{[k:k']}{2}+1}$. Writing it as a sum over a ``trace set'' or a ``norm set'' we should obtain (for generic $f$) an estimate of the form $C_dq'^{\frac{[k:k']+1}{2}}$, where $C_d$ depends only on $d$ and $[k:k']$. See the examples in sections 6 and 7 for explicit conditions on $f$ that imply this estimate. 

So we reduce the Frobenius trace sum of $\FF$ on $\AAA^1$ (or on $\GG_m$) to a sum of Frobenius traces of a simpler object $\GGG$ but on a more complicated space, defined by non-algebraic equations of the form $\Tr_{k/k'}(u)=\lambda$ or $\Nm_{k/k'}(u)=\mu$. Sums of this type have been previously studied in the literature (cf. \cite{katz1993estimates}, \cite{katz1995note}, \cite{chai2004character}, \cite{li2006character}) mainly using the method of Weil descent. This method consists of identifying the set of elements of $k$ with a given trace or norm over $k'$ with the set of rational points on a $([k:k']-1)$-dimensional variety over $k'$, and thus reducing the sum to a more classical sum over the rational points of a variety. 

One disadvantage of this method is that one may lose some information when identifying those two sets. As a rather crude example of this phenomenon, consider the sum of the constant $1$ over the set of elements of $k$ with norm $1$ over $k'$. This sum is obviously equal to $\frac{q-1}{q'-1}=1+q'+\cdots+q'^{n-1}$, where $n=[k:k']$. When applying Weil descent, the given set is identified with an $(n-1)$-dimensional torus over $k'$, where $n=[k:k']$. So its cohomology has dimension ${{n-1}\choose{j}}$ and weight $2j$ in degree $j+n-1$ for every $j=0,\ldots,n-1$, and we obtain an estimate $\sum_{j=0}^{n-1}{{n-1}\choose{j}}q'^j=(1+q')^{n-1}$ which is worse than the actual value $\sum_{j=0}^{n-1} q'^j $. See remark \ref{compare} for a more elaborated example of this issue. 

In this article we introduce another method to systematically study these kinds of sums. For a given $\ell$-adic sheaf (or, more generally, a derived category object) $\FF$ on a geometrically connected commutative algebraic group $G$ over $k$, an integer $m\geq 1$ and a point $t\in G(k_m)$ (where $k_m$ is the extension of $k$ of degree $m$ in a fixed algebraic closure $\bar k$) we define the $r$-th local \emph{norm $L$-function} of $\FF$ at $t$ as
$$
L^{\Nm,r}(\FF,k_m,t;T):=\exp\sum_{s\geq 1}f^{\Nm,r}_\FF(k_{ms},t)\frac{T^s}{s}
$$
where
$$
f^{\Nm,r}_\FF(k_m,t):=\sum_{\Nm_{k_{mr}/k_m}(u)=t}\Tr(\Frob_{k_{mr},u}|\FF_{\bar u})
$$
and $\Nm_{k_{mr}/k_m}:G(k_{mr})\to G(k_m)$ is the norm map.

These functions can be used to estimate sums defined on sets given by trace and norm conditions in the same way that classical $L$-functions are used to obtain information about usual sums over the set of rational points of a variety. The main result of this article is the fact that these functions are rational:

\begin{thm}
 For every object $\FF\in\Dbc(G,\QQ)$, every $m\geq 1$ and every $t\in G(k_m)$, the $r$-th norm $L$-function $L^{\Nm,r}(\FF,k_m,t;T)$ is rational. If $\FF$ is mixed of integral weights, all its reciprocal roots and poles are pure of integral $q^m$-weight.
\end{thm}

We give some explicit estimates in the cases where $G=\AAA^1_k$ or $G=\GG_{m,k}$. In order to obtain good estimates for the sums, we need information on the degree and the weights of the roots and poles of these functions. We will see that, in both cases, there are special objects (extensions of Artin-Schreier sheaves in the additive case, extensions of Kummer sheaves in the multiplicative case) for which the weights reach their maximal value. For these objects there are explicit formulas for the trace and norm $L$-functions, so they can be easily controlled. For the remaining objects, there are good estimates for the weights of the reciprocal roots and poles of the $L$-functions at $t$ for all $t$ in a certain dense open subset of $G$ that depends on $r$, which can be explicitely computed in some cases. In many examples we will also be able to obtain explicit bounds for the total degree of the $L$-functions.

The additive and multiplicative cases can be studied in parallel. However, in the additive case there is a great advantage thanks to the existence of the $\ell$-adic Fourier transform. This allows to reduce the study of the trace $L$-functions to the study of the Fourier transform of the object $\FF$, and more precisely of its geometric monodromy (section 6). In the multiplicative case we lack this shortcut, and instead we rely on recent work by Katz \cite{katz2010mellin} on the tensor category of perverse sheaves on $\GG_{m,k}$ under convolution in order to obtain explicit results for some important examples (section 7).

The author would like to thank the referee for his careful reading of an earlier version of the article and his many useful suggestions for improvement.

\section{$\QQ$-representable functions}

Let $k=\Fq$ be a finite field of characteristic $p>0$ and $\bar k=\FFq$ a fixed algebraic closure. For each positive integer $m$, we denote by $k_m={\mathbb F}_{q^m}$ the unique extension of $k$ of degree $m$ inside $\bar k$. Fix a prime $\ell\neq p$ and a field isomorphism $\iota:\QQ\to{\mathbb C}$. We will use this isomorphism to identify $\QQ$ and ${\mathbb C}$ without making any further mention to it. Let $X$ be a separated scheme of finite type over $k$. We define $\CC_X$ to be the set
$$
\CC_X:=\{f:\coprod_{m\geq 1} X(k_m)\to \QQ\}=\prod_{m\geq 1}\{f: X(k_m)\to\QQ\}
$$
of $\QQ$-valued functions defined on the disjoint union $\coprod_{m\geq 1}X(k_m)$. It is a commutative ring with the obvious point-wise operations.

Let ${\mathcal Sh}(X,\QQ)$ be the abelian category of constructible $\QQ$-sheaves on $X$, $\Dbc(X,\QQ)$ the corresponding derived category and $K_0(X,\QQ)$ its Grothendieck group. That is, the free abelian group generated by the isomorphism classes of elements of ${\mathcal Sh}(X,\QQ)$ with relations $[{\mathcal F}]=[{\mathcal G}]+[{\mathcal H}]$ for every short exact sequence $0\to{\mathcal G}\to{\mathcal F}\to{\mathcal H}\to 0$. It is also the Grothendieck group of $\Dbc(X,\QQ)$, and for every $K\in \Dbc(X,\QQ)$ we have $[K]=\sum_i(-1)^i[\HHH^i(K)]$.

From now on we will only consider sheaves and derived category objects which are mixed of integral $q$-weights (either with respect to the isomorphism $\iota$, or with respect to \emph{every} isomorphism $\QQ\to{\mathbb C}$). For every $\FF\in\Sha$ (or, more generally, in $\Dbc(X,\QQ)$) we define an element $f_\FF$ of $\CC_X$ in the following way (cf. \cite[1.1]{laumon1987transformation}): for every $m\geq 1$ and every $t\in X(k_m)$, $f_\FF(k_m,t):=\Tr(\Frob_{k_m,t}|\FF_{\bar t})$ is the trace of the action of a geometric Frobenius element at $t$ on the stalk of $\FF$ at a geometric point $\bar t$ over $t$. Given an exact sequence $0\to{\mathcal G}\to{\mathcal F}\to{\mathcal H}\to 0$ in $\Sha$ it is clear that $f_\FF=f_\GGG+f_\HHH$ (since taking stalks at a given geometric point is an exact functor), therefore the application $\FF\mapsto f_\FF$ extends to a homomorphism of abelian groups $\Phi:K_0(X,\QQ)\to\CC_X$, which is actually a homomorphism of rings if we endow $K_0(X,\QQ)$ with the multiplication defined by $[\FF]\times[\GGG]=[\FF\otimes\GGG]$ for sheaves $\FF$, $\GGG$ and extended by linearity. The homomorphism $\Phi$ is injective \cite[Th\'eor\`eme 1.1.2]{laumon1987transformation}.

\begin{defn}
 A function $f\in\CC_X$ is called \emph{$\QQ$-representable} if it is in the image of $\Phi$, that is, if there exists some (necessarily unique) $F\in K_0(X,\QQ)$ such that $f=f_F:=\Phi(F)$. In that case $f$ is said to be \emph{represented by} $F$. The set of all such functions is denoted by $\CC_{X,rep}$.
\end{defn}

Thus $\CC_{X,rep}$ is a subring of $\CC_X$ isomorphic to $K_0(X,\QQ)$. The following results are easy consequences of the definitions and Grothendieck's trace formula:

\begin{prop}\cite[1.1.1.4]{laumon1987transformation}
 For every $k$-morphism $\phi:X\to Y$ of separated schemes of finite type over $k$ and every $f\in\CC_{Y,rep}$ the function $\phi^\star f$ given by $\phi^\star f(k_m,t)=f(k_m,\phi(t))$ is in $\CC_{X,rep}$.
\end{prop}

\begin{prop}\cite[1.1.1.3]{laumon1987transformation}
For every $k$-morphism $\phi:X\to Y$ of separated schemes of finite type over $k$ and every $f\in\CC_{X,rep}$ the function $\phi_! f$ given by $\phi_! f(k_m,t)=\sum_{u\in X(k_m),\phi(u)=t}f(k_m,u)$ is in $\CC_{Y,rep}$.
\end{prop}

\begin{defn}
 For every $\ell$-adic unit $\alpha\in\QQ$ of integral $q$-weight, the \emph{constant function} $\kappa_\alpha\in\CC_X$ defined by $\alpha$ is given by $(k_m,t)\mapsto \alpha^m$. 
\end{defn}

The constant function $\kappa_\alpha$ is represented by the geometrically constant sheaf $\alpha^{\mathrm{deg}}$, that is, the pull-back of the character $\mathrm{Gal}(\bar k/k)\to\QQ^\star$ mapping the geometric Frobenius element to $\alpha$. The subset of constant functions is a multiplicative subgroup of $\CC_{X,rep}$, but it is not closed under addition.

\begin{defn}
 Let $f\in\CC_X$, $m\geq 1$ and $t\in X(k_m)$. The \emph{local $L$-function of $f$ at $t$} is defined as
$$
L(f,k_m,t;T):=\exp\sum_{s\geq 1} f(k_{ms},t)\frac{T^s}{s}\in 1+T\QQ[[T]]
$$ 
\end{defn}

It is clear that for every $f$ and $g$ in $\CC_X$ we have
$$
L(f+g,k_m,t;T)=L(f,k_m,t;T)\cdot L(g,k_m,t;T),
$$
for every $\ell$-adic unit $\alpha$ of integral $q$-weight
$$
L(\kappa_\alpha f,k_m,t;T)=L(f,k_m,t;\alpha^m T)
$$
and, for every $k$-morphism $\phi:X\to Y$
$$
L(\phi^\star f,k_m,t;T)=L(f,k_m,\phi(t);T).
$$

\begin{prop}
 For every $f\in\CC_{X,rep}$, every $m\geq 1$ and every $t\in X(k_m)$ the $L$-function $L(f,k_m,t;T)$ is rational and all its reciprocal roots and poles have integral $q^m$-weight.
\end{prop}

\begin{proof}
 By additivity, it suffices to prove if when $f$ is represented by a sheaf $\FF\in\Sha$. But in that case it is well known that
$$
L(f,k_m,t;T)=\exp\sum_{s\geq 1}\Tr(\Frob_{k_m,t}^s|\FF_{\bar t})\frac{T^s}{s}=\det(1-T\cdot\Frob_{k_m,t}|\FF_{\bar t})^{-1}.
$$
\end{proof}

\section{The convolution Adams operation}

Let $S=\Spec k$ be the spectrum of a field (or, more generally, a base scheme such that the derived category of $\ell$-adic sheaves is well defined on ${\mathcal Sch}/S$, e.g. a regular scheme of dimension $\leq 1$ \cite[1.1.2]{deligne1980conjecture}). Let $X$ be a separated scheme of finite type over $S$ and $H$ a finite group (regarded as acting trivially on $X$), and consider the category $\Sh(X,\QQ)_H$ of $\QQ$-sheaves on $X$ endowed with an action of $H$ and its derived category $\Dbc(X,\QQ)_H$.

Given a representation $\rho:H\to\mathrm{GL}(V)$ (where $V$ is a finite dimensional vector space over $\QQ$) we get a functor $\Sh(X,\QQ)_H\to\Sh(X,\QQ)$ given by $\FF\to\FF^\rho:={\mathcal Hom}_H(V,\FF)$, where $V$ is regarded as a constant sheaf on $X$ with an $H$-action given by $\rho$. If ${\mathbf 1}$ is the trivial representation then $\FF^{\mathbf 1}=\FF^H$ is the $H$-invariant part, and in general $\FF^\rho={\mathcal Hom}(V,\FF)^H$. Since $\QQ$ has characteristic zero, the functor $\FF\to\FF^\rho$ is exact, and it commutes with passage to fibres: for every geometric point $\bar x\in X(\bar k)$, we have $(\FF^\rho)_{\bar x}=\Hom_H(V,\FF_{\bar x})$. In particular, it extends to the derived category and we get a functor $\Dbc(X,\QQ)_H\to\Dbc(X,\QQ)$, $K\mapsto K^\rho$ such that $\HHH^i(K^\rho)=\HHH^i(K)^\rho$ for every $i\in{\mathbb Z}$. If $\rho':H\to\mathrm{GL}(V')$ is another finite dimensional $\QQ$-representation, it is clear that $K^{\rho\oplus\rho'}=K^\rho\oplus K^{\rho'}$.

\begin{lem}\label{directimage}
 Let $f:X\to Y$ be an $S$-morphism of separated schemes of finite type over $S$ and $K\in\Dbc(X,\QQ)$ an object with an $H$-action. Then for every finite dimensional representation $\rho$ of $H$ $(\R f_\star K)^\rho=\R f_\star (K^\rho)$ and $(\R f_! K)^\rho=\R f_! (K^\rho)$.
\end{lem}

\begin{proof}
 It suffices to prove it for $\R f_\star$, since clearly $j_! (K^\rho)=(j_! K)^\rho$ for an open immersion $j$. Let $\FF$ be a $\QQ$-sheaf on $X$ with an $H$-action. Since $f_\star$ is left exact, we have $(f_\star\FF)^H=f_\star(\FF^H)$. Therefore for every representation $\rho:H\to\GL(V)$ of $H$
$$
(f_\star\FF)^\rho={\mathcal Hom}_H(V,f_\star\FF)={\mathcal Hom}(V,f_\star\FF)^H=(f_\star{\mathcal Hom}(f^\star V,\FF))^H=
$$
$$
=f_\star({\mathcal Hom}(V,\FF)^H)=f_\star({\mathcal Hom}_H(V,\FF))=f_\star(\FF^\rho).
$$

By the exactness of $(-)^\rho$ we deduce that $(\R f_\star\FF)^\rho=\R f_\star(\FF^\rho)$ for every sheaf $\FF$, and then also for every object $K\in\Dbc(X,\QQ)$.
\end{proof}

We will be mainly interested in the following situation: $K\in\Dbc(X,\QQ)$ is any object, $r\geq 1$ is an integer, and the symmetric group in $r$ letters ${\mathfrak S}_r$ acts on $K^{\otimes r}$ via permutation of the factors. Then for every representation $\rho:{\mathfrak S}_r\to\GL(V)$ we get an object $\R(\rho)K:=(K^{\otimes r})^\rho\in\Dbc(X,\QQ)$. We write $\Sym^r K$ (respectively $\wedge^r K$) for $\R(\rho)K$ if $\rho$ is the trivial representation (resp. the sign character).

If $G$ is a geometrically connected commutative group scheme of finite type over $S$ we will also use the following variant. Recall that, for any two objects $K,L\in\Dbc(G,\QQ)$, their \emph{$!$-convolution} (which we will simply call convolution) is the object 
$K\ast L :=\R\mu_!(K\boxtimes L)\in\Dbc(G,\QQ)$, where $\mu:G\times G\to G$ is the multiplication map. It is an associative and commutative triangulated bifunctor \cite[2.5]{katz1996rls}. 

For any $r\geq 1$, the multiplication map $G^r\to G$ factors through $\Sym^r G$, so $K^{\ast r}:=K\ast\cdots\ast K$ ($r$ factors) is endowed with a natural action of ${\mathfrak S}_r$, induced by its action on $\pi_!K^{\boxtimes r}$ above $\Sym^rG$ by permutation of the factors (where $\pi:G^r\to\Sym^r G$ is the natural projection). We denote $\R(\ast\rho)K:=(K^{\ast r})^\rho\in\Dbc(G,\QQ)$, and we write $\Sym^{\ast r} K$ (respectively $\wedge^{\ast r} K$) for $\R(\ast\rho)K$ if $\rho$ is the trivial representation (resp. the sign character).

The following result generalizes \cite[Lemme 1.3]{deligne569fonction} (for $A$ a field of characteristic $0$):

\begin{prop}\label{mainprop}
 Let $a:G\to S$ be the structural map. Then for every $K\in\Dbc(G,\QQ)$ and every finite dimensional $\QQ$-representation $\rho$ of ${\mathfrak S}_r$ we have a quasi-isomorphism
$$
\R a_!(\R(\ast\rho)K)\cong\R(\rho)(\R a_!K).
$$
In particular, if $S=\Spec k$ is the spectrum of a separably closed field we have $$\R\Gamma_c(G,\R(\ast\rho)K)\cong\R(\rho)(\R\Gamma_c(G,K)).$$
\end{prop}

\begin{proof}
 By lemma \ref{directimage} we have $\R a_!(\R(\ast\rho)K)=\R a_!((K^{\ast r})^\rho)\cong(\R a_!(K^{\ast r}))^\rho$. If $\mu:G^r\to G$ denotes the multiplication map, $K^{\ast r}=\R\mu_! K^{\boxtimes r}$, so 
$$
(\R a_!(K^{\ast r}))^\rho\cong(\R a_!(\R\mu_!K^{\boxtimes r}))^\rho\cong(\R(a\mu)_! K^{\boxtimes r})^\rho.
$$

Now by K\"unneth, there is a ${\mathfrak S}_r$-equivariant quasi-isomorphism $\R(a\mu)_! K^{\boxtimes r}\cong(\R a_! K)^{\otimes r}$ (where ${\mathfrak S}_r$ acts on the right by permutation of the factors), so
$$
(\R(a\mu)_! K^{\boxtimes r})^\rho\cong ((\R a_! K)^{\otimes r})^\rho=\R(\rho)(\R a_! K).
$$
\end{proof}

We also have the following shift formulas:

\begin{prop}\label{shift}
 Let $\sigma$ be the sign character and $\rho$ any finite dimensional representation of ${\mathfrak S}_r$. For every $K\in\Dbc(X,\QQ)$ we have
$$
\R(\rho)(K[1])\cong(\R(\rho\otimes\sigma)K)[r]
$$
and for every $K\in\Dbc(G,\QQ)$
$$
\R(\ast\rho)(K[1])\cong(\R(\ast(\rho\otimes\sigma))K)[r].
$$
\end{prop}

\begin{proof}
 We will prove the first formula, the second one is similar. With the obvious notations, we have
$$
\R(\rho)(K[1])\cong{\mathcal Hom}_{{\mathfrak S}_r}(\rho,K[1]^{\otimes r})
$$
and
$$
(\R(\rho\otimes\sigma)K)[r]\cong{\mathcal Hom}_{{\mathfrak S}_r}(\rho\otimes\sigma,K^{\otimes r}[r])
$$
so it suffices to show that the action of ${\mathfrak S}_r$ on $K^{\otimes r}[r]$ is the same as its action on $K[1]^{\otimes r}(\cong K^{\otimes r}[r])$ twisted by $\sigma$.

Suppose that $r=2$, and let $\tau\in{\mathfrak S}_2$ be the transposition. There is a natural isomorphism $\phi:K[1]\otimes K[1]\cong (K\otimes K)[2]$ given by $a_i\otimes b_j\mapsto (-1)^{i}a_i\otimes b_j$ for $a_i\in K^i=K[1]^{i-1}$ and $b_j\in K^j=K[1]^{j-1}$. On $K[1]\otimes K[1]$, the transposition $\tau$ acts by $a_i\otimes b_j\mapsto (-1)^{(i-1)(j-1)}b_j\otimes a_i$ (since $a_i,b_j$ live on degrees $i-1$ and $j-1$ respectively). On $(K\otimes K)[2]$ it acts by $a_i\otimes b_j\mapsto (-1)^{ij}b_j\otimes a_i$. So for $a_i\otimes b_j\in K[1]^{i-1}\otimes K[1]^{j-1}$ we have $\phi\circ\tau(a_i\otimes b_j)=\phi((-1)^{(i-1)(j-1)}b_j\otimes a_i)=(-1)^{ij-i+1}b_j\otimes a_i$ and $\tau\circ\phi(a_i\otimes b_j)=\tau((-1)^ia_i\otimes b_j)=(-1)^{i+ij}b_j\otimes a_i$, so $\phi\circ\tau(a_i\otimes b_j)=-\tau\circ\phi(a_i\otimes b_j)$. In other words, the action of $\tau$ on $K[1]\otimes K[1]$ is the negative of its action on $(K\otimes K)[2]$ via the isomorphism $\phi$.

Now let $r>2$, and let $\tau\in{\mathfrak S}_r$ be any transposition. Without loss of generality, we may assume that $\tau=(1\; 2)$. Then we have an isomorphism $\phi_r:K[1]^{\otimes r}=K[1]^{\otimes 2}\otimes K[1]^{\otimes(r-2)}\cong K^{\otimes 2}\otimes K^{\otimes (r-2)}[r]=K^{\otimes r}[r]$, where $\phi_r=\phi\otimes\phi_{r-2}$ inductively. Since $\tau$ acts trivially on $K[1]^{\otimes (r-2)}$ and $K^{\otimes (r-2)}$, by the $r=2$ case the actions of $\tau$ on $K[1]^{\otimes r}$ and $K^{\otimes r}[r]$ differ by sign. We conclude that the actions of ${\mathfrak S}_r$ on $K[1]^{\otimes r}$ and $K^{\otimes r}[r]$ are twists of each other by the sign character.
\end{proof}

If $\rho':{\mathfrak S}_r\to V'$ is another finite dimensional $\QQ$-representation of ${\mathfrak S}_r$, it is clear that $\R(\rho\oplus\rho')K=\R(\rho)K\oplus\R(\rho')K$ and $\R(\ast(\rho\oplus\rho'))K=\R(\ast\rho)K\oplus\R(\ast\rho')K$. Therefore it makes sense to define $\R(\tau)K$ (respectively $\R(\ast\tau)K$) as an element of the Grothendieck group $K_0(X,\QQ)$ (resp. $K_0(G,\QQ)$) for any virtual $\QQ$-representation $\tau$ of ${\mathfrak S}_r$: if $\tau=\rho-\rho'$, we set $\R(\tau)K:=[\R(\rho)K]-[\R(\rho')K]$ (resp. $\R(\ast\tau)K:=[\R(\ast\rho)K]-[\R(\ast\rho')K]$). Proposition \ref{mainprop} implies that 
\begin{equation}\label{adamscommute}
\R a_!(\R(\ast\tau)K)=\R(\tau)(\R a_!K)
\end{equation}
in $K_0(S,\QQ)$ for $a:G\to S$ the structural map.

Let $\rho:{\mathfrak S}_r\to \mathrm{GL}(L)$ be the standard representation, where $L\subset\QQ^r$ is the hyperplane defined by $\sum x_i=0$, and let $\tau_r$ be the virtual representation $\sum_{i=0}^{r-1} (-1)^i\wedge^i\rho$. 

\begin{prop}\label{young}
 Let $\FF\in\Sha$. Then
$$
\R(\tau_r)\FF=\FF^{[r]}:=\sum_{i=1}^r(-1)^{i-1} i [\Sym^{r-i}\FF\otimes\wedge^i\FF]\in K_0(X,\QQ)
$$
is the $r$-th Adams power of $\FF$.
\end{prop}

\begin{proof}
 For every $i=1,\ldots,r$ we have 
$$
\Sym^{r-i}\FF\otimes\wedge^ i\FF ={\mathcal Hom}_{{\mathfrak S}_{r-i}\times{\mathfrak S}_i}({\mathbf 1}\times\sigma,\FF^{\otimes r})
$$
where $\sigma$ is the sign character of ${\mathfrak S}_i$. By Frobenius reciprocity, this is the same as $\R(\eta)\FF$, where $\eta=\mathrm{Ind}^{{\mathfrak S}_r}_{{\mathfrak S}_{r-i}\times{\mathfrak S}_i}({\mathbf 1}\times\sigma)$.

By Pieri's formula (cf. \cite[4.41, A.7]{fulton1996representation}) for $i<r$ $\eta$ is the sum of two irreducible representations $\eta_i$ and $\eta_{i+1}$ with inverted-L shaped Young diagrams
$$
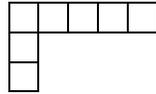
\yng(5,1,1)
$$
of vertical lengths $i$ and $i+1$ respectively, that is, the $(i-1)$-th and $i$-th exterior powers of the standard representation of ${\mathfrak S}_r$ (cf. \cite[4.6]{fulton1996representation}). Therefore
$$
\sum_{i=1}^r(-1)^{i-1} i[\Sym^{r-i}\FF\otimes\wedge^ i\FF]=\R\left(\left(\sum_{i=1}^{r-1}(-1)^{i-1} i (\wedge^{i-1}\rho+\wedge^{i}\rho)\right)+(-1)^{r-1} r\sigma\right)\FF=
$$
$$
=\R\left(\sum_{i=0}^{r-1} (-1)^{i} \wedge^i\rho\right)\FF=\R(\tau_r)\FF.
$$
\end{proof}

Since the Adams power is additive, we deduce

\begin{cor}
 The map $\Sh(X,\QQ)\to K_0(X,\QQ)$ given by $\FF\mapsto \R(\tau_r)\FF$ extends to a homomorphism of abelian groups $\R(\tau_r):K_0(X,\QQ)\to K_0(X,\QQ)$.
\end{cor}

\begin{rem}\label{rank}\emph{
 The sum $\sum_{i=0}^{r-1}(-1)^i\R(\wedge^i\rho)\FF$ gives an ``optimal'' expression for the Adams power, in the sense that there is no further cancellation among different sign terms. If $\FF$ has rank $n$ then for every $i=0,\ldots,r-1$ the rank of  $\R(\wedge^i\rho)\FF$ is the dimension of the Weyl module corresponding to the partition $(r-i,1,\cdots,1)$ (with $i$ $1$'s), that is,
$$
{{n+r-i-1}\choose{r}}{{r-1}\choose{i}}
$$
by \cite[Theorem 6.3, 6.4]{fulton1996representation}.}
\end{rem}

\begin{defn}
 Let $K\in \Dbc(G,\QQ)$. The $r$-th \emph{convolution Adams power} of $K$ is the object $K^{[\ast r]}:=\R(\ast\tau_r)K\in K_0(G,\QQ)$.
\end{defn}

\begin{prop}
 If $K\to M\to L\to K[1]$ is a distinguished triangle in $\Dbc(G,\QQ)$ then $M^{[\ast r]}=K^{[\ast r]}+L^{[\ast r]}$ for every $r\geq 1$. In particular, the $r$-th convolution Adams power extends to a homomorphism of abelian groups $K_0(G,\QQ)\to K_0(G,\QQ)$. 
\end{prop}

\begin{proof}
 For every $r\geq 1$ $M^{\ast r}=\R\mu_!(M^{\boxtimes r})$ has a filtration with quotients $\R\mu_! P_k$ for $k=0,1,\ldots,r$, where
$$
P_k=\bigoplus_{J\subseteq\{1,\ldots,r\},|J|=k}N_{J,1}\boxtimes\cdots\boxtimes N_{J,r}
$$
and $N_{J,j}=K$ (respectively $N_{J,j}=L$) if $j\in J$ (resp. $j\notin J$). The action of ${\mathfrak S}_r$ preserves this filtration, and acts transitively on the set $\{\R\mu_!(N_{J,1}\boxtimes\cdots\boxtimes N_{J,r})|J\subseteq\{1,\ldots,r\},|J|=k\}$ for each $k$, with stabilizer ${\mathfrak S}_k\times{\mathfrak S}_{r-k}$ for $\R\mu_!(K^{\boxtimes k}\boxtimes L^{\boxtimes(r-k)})$. In other words, the action of ${\mathfrak S}_r$ on $\R\mu_!P_k$ is the action induced by that of ${\mathfrak S}_k\times{\mathfrak S}_{r-k}$ on $\R\mu_!(K^{\boxtimes k}\boxtimes L^{\boxtimes(r-k)})=K^{\ast k}\ast L^{\ast(r-k)}$ by permutation of the first $k$ and the last $r-k$ factors. Therefore
$$
[\R(\ast\rho)M]=[{\mathcal Hom}_{{\mathfrak S}_r}(\rho,M^{\ast r})]=\sum_{k=0}^r[{\mathcal Hom}_{{\mathfrak S}_r}(\rho,\mathrm{Ind}^{{\mathfrak S}_r}_{{\mathfrak S}_k\times{\mathfrak S}_{r-k}}K^{\ast k}\ast L^{\ast(r-k)})]=
$$
$$
=\sum_{k=0}^r[{\mathcal Hom}_{{\mathfrak S}_k\times{\mathfrak S}_{r-k}}(\rho_{|{\mathfrak S}_k\times{\mathfrak S}_{r-k}},K^{\ast k}\ast L^{\ast(r-k)})]
$$
for any finite dimensional $\QQ$ representation $\rho$ of ${\mathfrak S}_r$. In particular, if $\rho$ is the standard representation and $1\leq k\leq r-1$, the Littlewood-Richardson formula gives (cf. \cite[4.43, A.8]{fulton1996representation}):
$$
(\wedge^i\rho)_{|{\mathfrak S}_k\times{\mathfrak S}_{r-k}}=\bigoplus_{j+l=i}(\wedge^j\rho_1)\times(\wedge^l\rho_2)\oplus\bigoplus_{j+l=i-1}(\wedge^j\rho_1)\times(\wedge^l\rho_2)
$$
where $\rho_1$ and $\rho_2$ are the standard representations of ${\mathfrak S}_k$ and ${\mathfrak S}_{r-k}$ respectively and $j\leq k-1$, $l\leq r-k-1$ in the sums, so we obtain
$$
M^{[\ast r]}=\sum_{i=0}^{r-1}(-1)^i[\R(\ast\wedge^i\rho)M]=
$$
$$
=\sum_{i=0}^{r-1}(-1)^i\sum_{k=0}^r[{\mathcal Hom}_{{\mathfrak S}_k\times{\mathfrak S}_{r-k}}((\wedge^i\rho)_{|{\mathfrak S}_k\times{\mathfrak S}_{r-k}},K^{\ast k}\ast L^{\ast(r-k)})]=
$$
$$
=\sum_{i=0}^{r-1}(-1)^i[\R(\ast\wedge^i\rho)K]+\sum_{i=0}^{r-1}(-1)^i[\R(\ast\wedge^i\rho)L]+
$$
$$
+\sum_{k=1}^{r-1}\sum_{i=0}^{r-1}(-1)^i\left(\sum_{j+l=i}[{\mathcal Hom}_{{\mathfrak S}_k\times{\mathfrak S}_{r-k}}((\wedge^j\rho_1)\times(\wedge^l\rho_2),K^{\ast k}\ast L^{\ast(r-k)})]+\right.
$$
$$
\left.+\sum_{j+l=i-1}[{\mathcal Hom}_{{\mathfrak S}_k\times{\mathfrak S}_{r-k}}((\wedge^j\rho_1)\times(\wedge^l\rho_2),K^{\ast k}\ast L^{\ast(r-k)})]\right).
$$

The last sum clearly vanishes, so we conclude that
$$
M^{[\ast r]}=\sum_{i=0}^{r-1}(-1)^i[\R(\ast\wedge^i\rho)K]+\sum_{i=0}^{r-1}(-1)^i[\R(\ast\wedge^i\rho)L]=K^{[\ast r]}+L^{[\ast r]}.
$$
\end{proof}

\section{Norm $L$-functions}

We go back to the case where $k=\Fq$ is a finite field. Let $G$ be a geometrically connected commutative group scheme of finite type over $k$. For every positive integers $m,r$ there is a norm map $\Nm_{k_{mr}/k_m}:G(k_{mr})\to G(k_m)$ given by
$$
\Nm_{k_{mr}/k_m}(u)=\prod_{\sigma\in\Gal(k_{mr}/k_m)}\sigma(u).
$$

\begin{defn}
 Let $f\in\CC_G$ and $r\geq 1$ an integer. The \emph{$r$-th norm power} of $f$ is the function $f^{\Nm,r}\in\CC_G$ given by $$f^{\Nm,r}(k_m,t)=\sum_{\Nm_{k_{mr}/k_m}(u)=t}f(k_{mr},u)$$
\end{defn}

The following properties are immediate consequences of the definitions:

\begin{prop}\label{easy}
 Let $f,g\in\CC_G$, $\alpha\in\QQ$ an $\ell$-adic unit of integral $q$-weight and $r\geq 1$ an integer. Then
\begin{enumerate}
 \item $(f+g)^{\Nm,r}=f^{\Nm,r}+g^{\Nm,r}$
\item $(\kappa_\alpha\cdot f)^{\Nm,r}=\kappa_{\alpha^r}\cdot f^{\Nm,r}$
\end{enumerate}
\end{prop}

The goal of this section is to show that $\CC_{G,rep}$ is invariant under these operations. More precisely, the Frobenius trace function of the $r$-th convolution Adams power of $K$ is the $r$-th norm power of the Frobenius trace function of $K$:

\begin{thm}
 For every $K\in K_0(G,\QQ)$ and every $r\geq 1$ we have
$$
\Phi(K^{[\ast r]})=\Phi(K)^{\Nm,r}.
$$
\end{thm}

Let $\chi:G(k)\to \QQ^\star$ be a character. By \cite[1.4-1.8]{deligne569application} there is a rank $1$ smooth $\QQ$-sheaf $\LL_\chi$ on $G$ such that, for every $m\geq 1$ and every $t\in G(k_m)$,
$$
\Tr(\Frob_{k_m,t}|\LL_{\chi,\bar t})=\chi(\Nm_{k_m/k}(t)).
$$
Tensoring with $\LL_\chi$ is an autoequivalence of the triangulated category $\Dbc(G,\QQ)$. By \cite[8.1.10]{katz1990esa} there is a quasi-isomorphism $K^{\ast r}\otimes\LL_\chi\cong(K\otimes\LL_\chi)^{\ast r}$ for every $r\geq 1$, which is compatible with the natural ${\mathfrak S}_r$ actions. In particular, for every (virtual) finite dimensional $\QQ$-representation $\rho$ of ${\mathfrak S}_r$ we have
$$
(\R(\ast\rho)K)\otimes\LL_\chi\cong\R(\ast\rho)(K\otimes\LL_\chi).
$$
Taking $\rho=\tau_r$, we get
\begin{equation}\label{tensorchi}
 K^{[\ast r]}\otimes\LL_\chi=(K\otimes\LL_\chi)^{[\ast r]}
\end{equation}
in $K_0(G,\QQ)$ for every $r\geq 1$.

\begin{proof}[Proof of theorem 4.3]
 We have to show that, for every $m\geq 1$ and every $t\in G(k_m)$,
$$
\Phi(K^{[\ast r]})(k_m,t)=\sum_{\Nm_{k_{mr}/k_m}(u)=t}\Phi(K)(k_{mr},u).
$$
Since the operation $K\mapsto K^{[\ast r]}$ commutes with extending scalars to a finite extension of $k$, we may assume without loss of generality that $m=1$. The (ordinary finite group) Fourier transform gives a bijection between the set of $\QQ$-valued maps defined on $G(k)$ and the set of $\QQ$-valued maps defined on the set of characters of $G(k)$, so the previous equality for every $t\in G(k)$ is equivalent to the equality
$$
\sum_{t\in G(k)}\chi(t)\Phi(K^{[\ast r]})(k,t)=\sum_{t\in G(k)}\chi(t)\sum_{\Nm_{k_{r}/k}(u)=t}\Phi(K)(k_r,u).
$$
for every character $\chi:G(k)\to\QQ^\star$. The left hand side is, by the Grothendieck-Lefschetz trace formula,
$$
\sum_{t\in G(k)}\Phi(\LL_\chi)(k,t)\Phi(K^{[\ast r]})(k,t)=\sum_{t\in G(k)}\Phi(K^{[\ast r]}\otimes\LL_\chi)(k,t)=
$$
$$
=\sum_{t\in G(k)}\Phi((K\otimes\LL_\chi)^{[\ast r]})(k,t)=\Tr(\Frob_k|\R\Gamma_c(G\otimes\bar k,(K\otimes\LL_\chi)^{[\ast r]})).
$$

The right hand side is, again by the trace formula,
$$
\sum_{u\in G(k_r)}\chi(\Nm_{k_r/k}(u))\Phi(K)(k_r,u)=\sum_{u\in G(k_r)}\Phi(K\otimes\LL_\chi)(k_r,u)=
$$
$$
=\Tr(\Frob_{k_r}|\R\Gamma_c(G\otimes\bar k,K\otimes\LL_\chi))=\Tr(\Frob_k^r|\R\Gamma_c(G\otimes\bar k,K\otimes\LL_\chi))=
$$
$$
=\Tr(\Frob_k|\R\Gamma_c(G\otimes\bar k,K\otimes\LL_\chi)^{[r]})
$$
since $\Tr(\phi|V^{[r]})=\Tr(\phi^r|V)$ for any endomorphism $\phi$ of a vector space $V$ (cf. \cite[Theorem 1.1]{fu2004moment}). The result follows then from equation (\ref{adamscommute}) applied to $\tau=\tau_r$, which tells us that the virtual $\Frob_k$-modules $\R\Gamma_c(G\otimes\bar k,(K\otimes\LL_\chi)^{[\ast r]})$ and $\R\Gamma_c(G\otimes\bar k,K\otimes\LL_\chi)^{[r]}$ are isomorphic.
\end{proof}

\begin{cor}\label{maincor}
 For every $K\in K_0(G,\QQ)$, every $r\geq 1$ and every $t\in G(k_r)$ the $r$-th norm $L$-function of $K$ at $t$ $L^{\Nm,r}(K,k_m,t;T)$ is rational and all its reciprocal roots and poles have integral $q^m$-weight.
\end{cor}

Using proposition \ref{easy} and the injectivity of $\Phi$ we deduce

\begin{cor}\label{twist}
 For every $\ell$-adic unit $\alpha$ of integral $q$-weight we have
$$
(\alpha^{deg}\otimes K)^{[\ast r]}=\alpha^{r\cdot deg}\otimes K^{[\ast r]}.
$$
\end{cor}

We conclude this section with a useful formula relating the usual Adams powers and the convolution Adams powers in the case where $G$ is the additive group $\AAA^n_k$. Fix a non-trivial character $\psi:k\to\QQ^\star$. Recall that the \emph{Fourier transform} with respect to $\psi$ is the functor $\Dbc(\AAA^n_k,\QQ)\to\Dbc(\AAA^n_k,\QQ)$ given by (cf. \cite{katz1985transformation}):
$$
\FT_\psi(K)=\R\pi_{2!}(\pi_1^\star K\otimes\mu^\star \LL_\psi)[n]
$$
where $\pi_1,\pi_2:\AAA^n_k\times\AAA^n_k\to\AAA^n_k$ are the projections, $\mu:\AAA^n_k\times\AAA^n_k\to\AAA^1_k$ is given by $\mu((x_1,\ldots,x_n),(y_1,\ldots,y_n))=x_1y_1+\cdots+x_ny_n$ and $\LL_{\psi}$ is the Artin-Schreier smooth sheaf on $\AAA^1_k$ associated to $\psi$. It is an autoequivalence of the triangulated category $\Dbc(\AAA^n_k,\QQ)$. In particular, any action of a finite group $H$ on an object $K\in\Dbc(\AAA^n_k,\QQ)$ induces an action on $\FT_\psi(K)$.

\begin{lem}\label{lemmafourier}
 Let $K\in\Dbc(\AAA^n_k,\QQ)$ be an object with an action of the finite group $H$ and $\rho:H\to\GL(V)$ a finite dimensional representation of $H$. Then $(\FT_\psi(K))^\rho=\FT_\psi(K^\rho)$.
\end{lem}

\begin{proof}
 By lemma \ref{directimage} we have
$$
(\FT_\psi(K))^\rho=\R\pi_{2!}(\pi_1^\star K\otimes\mu^\star \LL_\psi)^\rho[n]=\R\pi_{2!}((\pi_1^\star K\otimes\mu^\star \LL_\psi)^\rho)[n]=
$$
$$
=\R\pi_{2!}((\pi_1^\star K)^\rho\otimes\mu^\star \LL_\psi)[n]
$$
since $\mu^\star\LL_\psi$ is smooth of rank $1$ on $\AAA^n_k\times\AAA^n_k$. On the other hand,
$$
(\pi_1^\star K)^\rho={\mathcal Hom}_\rho(V,\pi_1^\star K)={\mathcal Hom}(\pi_1^\star V,\pi_1^\star K)^H=
$$
$$
=(\pi_1^\star{\mathcal Hom}(V,K))^H=\pi_1^\star({\mathcal Hom}(V,K)^H)=\pi_1^\star(K^\rho)
$$
since $\pi_1^\star$ is exact. We conclude that
$$
(\FT_\psi(K))^\rho=\R\pi_{2!}(\pi_1^\star (K^\rho)\otimes\mu^\star \LL_\psi)[n]=\FT_\psi(K^\rho).
$$
\end{proof}

\begin{prop}\label{fourier}
 For every $K\in\Dbc(\AAA^n_k,\QQ)$, every integer $r\geq 1$ and every finite dimensional $\QQ$-representation $\rho$ of ${\mathfrak S}_r$ there is a quasi-isomorphism
$$
\FT_\psi(\R(\ast\rho)K)\cong\R(\rho\otimes\sigma^{n})\FT_\psi(K)[-n(r-1)].
$$
where $\sigma$ is the sign character of ${\mathfrak S}_r$. That is, the Fourier transform interchanges the operations $\R(\ast\rho)$ and $\R(\rho\otimes\sigma^{n})$ (up to a shift).
\end{prop}

\begin{proof}
 By \cite[Corollaire 9.6]{brylinski1986transformations}, for every $K,L\in\AAA^n_k$ we have the formula
$$
\FT_\psi(K\ast L)\cong\FT_\psi(K)\otimes\FT_\psi(L)[-n]
$$
and, in particular, there is a natural quasi-isomorphism
$$
\FT_\psi(K^{\ast r})\cong\FT_\psi(K)^{\otimes r}[-(r-1)n]
$$
for every $r\geq 1$. The natural ${\mathfrak S}_r$-actions on $\FT_\psi(K^{\ast r})[(r-1)n]$ and $\FT_\psi(K)^{\otimes r}$ differ by a twist by the $n$-th power of the sign character $\sigma$ (at a geometric point $\bar t$ over $t=(t_1,\ldots,t_n)\in k_m^n$ the stalks are $\R\Gamma_c(\AAA^n_{\bar k},K^{\ast r}\otimes\LL_{\psi_t})[rn]\cong\R\Gamma_c(\AAA^n_{\bar k},(K\otimes\LL_{\psi_t})^{\ast r})[rn]\cong\R\Gamma_c(\AAA^n_{\bar k},K\otimes\LL_{\psi_t})^{\otimes r}[rn]$ and $\R\Gamma_c(\AAA^n_{\bar k},K\otimes\LL_{\psi_t})[n]^{\otimes r}$ respectively, where $\LL_{\psi_t}$ is the rank $1$ smooth sheaf corresponding to the character of $k_m^n$ $x\mapsto \psi(\Tr_k(t_1x_1+\cdots+t_nx_n))$, and the proof of proposition \ref{shift} applied $n$ times shows that the actions of $\mathfrak S_r$ on these differ by a twist by $\sigma^n$). For any finite dimensional $\QQ$-representation $\rho$ of ${\mathfrak S}_r$ we have then
$$
\FT_\psi(\R(\ast\rho)K)=\FT_\psi((K^{\ast r})^\rho)=\FT_\psi(K^{\ast r})^\rho\cong(\FT_\psi(K)^{\otimes r}[-(r-1)n])^{\rho\otimes\sigma^n}=
$$
$$
=(\FT_\psi(K)^{\otimes r})^{\rho\otimes\sigma^{n}}[-(r-1)n]=\R(\rho\otimes\sigma^{n})\FT_\psi(K)[-n(r-1)]
$$
by lemma \ref{lemmafourier}.
\end{proof}

\begin{cor}\label{fourieradams}
 For every $K\in\Dbc(\AAA^n_k,\QQ)$ and every integer $r\geq 1$ we have
$$
\FT_\psi(K^{[\ast r]})=(\FT_\psi K)^{[r]}.
$$
\end{cor}

\begin{proof}
 Let $\rho$ be the standard representation of ${\mathfrak S}_r$. By the proposition, we have
$$
\FT_\psi(K^{[\ast r]})=\sum_{i=0}^{r-1}(-1)^i[\FT_\psi(\R(\ast\wedge^i\rho)K)]=
$$
$$
=(-1)^{n(r-1)}\sum_{i=0}^{r-1}(-1)^i[\R((\wedge^i\rho)\otimes\sigma^{n})\FT_\psi K].
$$
If $n$ is even this proves the statement. If $n$ is odd, then $\sigma^{n}=\sigma$ and $(\wedge^i\rho)\otimes\sigma=\wedge^{r-1-i}\rho$ so
$$
\FT_\psi(K^{[\ast r]})=(-1)^{r-1}\sum_{i=0}^{r-1}(-1)^i[\R(\wedge^{r-1-i}\rho)\FT_\psi K]=
$$
$$
=\sum_{i=0}^{r-1}(-1)^{r-1-i}[\R(\wedge^{r-1-i}\rho)\FT_\psi K]=\sum_{i=0}^{r-1}(-1)^{i}[\R(\wedge^{i}\rho)\FT_\psi K]=(\FT_\psi K)^{[r]}.
$$
\end{proof}

\section{The dimension 1 case}

From now on we will assume that $G$ is affine of dimension $1$ (so $G\otimes\bar k$ is either the affine line $\AAA^1_{\bar k}$ or the torus $\GG_{m,\bar k}$). We will describe more precisely the operation $K\mapsto K^{[\ast r]}$ in this situation by splitting $K$ into its perverse cohomology sheaves. Since $G$ is a smooth curve, perverse sheaves have an easy description \cite[5.2.2]{beilinson1982faisceaux}: they are objects $\PPP\in\Dbc(G,\QQ)$ which have non-zero cohomology only in degrees $0$ and $-1$, $\HHH^0(\PPP)$ is punctual (that is, $j^\star\HHH^0(\PPP)=0$ for some dense open set $j:U\hookrightarrow G$) and $\HHH^{-1}(\PPP)$ has no punctual sections (that is, the adjunction map $\HHH^{-1}(\PPP)\to j_\star j^\star\HHH^{-1}(\PPP)$ is injective for any dense open set $j:U\hookrightarrow G$). The full subcategory $\Perv(G,\QQ)\subset\Dbc(G,\QQ)$ of perverse sheaves on $G$ is an abelian category in which exact sequences are just distinguished triangles. Irreducible objects in this category are of two types: punctual objects $i_{x\star}\FF[0]$ where $i_x:\{x\}\hookrightarrow G$ is the inclusion of a closed point, and middle extensions $j_{\star !}(\FF[1])\cong (j_\star\FF)[1]$ where $j:U\hookrightarrow G$ is the inclusion of a dense open subset and $\FF$ is an irreducible smooth $\QQ$-sheaf on $U$. If $G=\AAA^1_k$, the Fourier transform functor (with respect to any non-trivial character $\psi:k\to\QQ^\star$) preserves perverse objects and is an autoequivalence of the category of perverse sheaves \cite[Corollaire 2.1.5]{katz1985transformation}.

 The derived category of the category of perverse sheaves is again $\Dbc(G,\QQ)$. In particular, its Grothendieck group is $K_0(G,\QQ)$. Therefore, by additivity it suffices to study $\PPP^{[\ast r]}$ for $\PPP$ an irreducible perverse sheaf. Since we are assuming that everything is mixed of integral $q$-weights, such a perverse sheaf is pure of some integral $q$-weight \cite[Corollaire 5.3.4]{beilinson1982faisceaux} and, in particular, geometrically semisimple \cite[Th\'eor\`eme 5.3.8]{beilinson1982faisceaux}. By corollary \ref{twist} we can further assume that it is pure of weight $1$.

\begin{lem}\label{negligible}
 Let ${\mathcal P}\in\Perv(G,\QQ)$ be irreducible. The following conditions are equivalent:
\begin{enumerate}
 \item $\PPP\otimes\bar k$ contains a sub-object isomorphic to $\LL_\chi[1]$ for some $r\geq 1$ and some character $\chi:G(k_r)\to\QQ^\star$ (cf. \cite[1.4-1.8]{deligne569application}).
 \item ${\mathcal P}\otimes\bar k$ is a direct sum of objects of the form $\LL_{\chi_i}[1]$ for some $r\geq 1$ and some characters $\chi_i:G(k_r)\to\QQ^\star$.
\end{enumerate}
In that case, if $r$ is the smallest positive integer such that $\LL_\chi$ is defined over $k_r$, then $\PPP\cong\alpha^{deg}\otimes\pi_{r\star}\LL_\chi[1]$ for some $\ell$-adic unit $\alpha$, where $\pi_r:G\otimes k_r\to G$ is the projection.
\end{lem}

\begin{proof}
(2)$\Rightarrow$(1) is trivial. Suppose that (1) holds, and let $r$ be the smallest positive integer such that $\LL_\chi$ is defined over $k_r$. Then we have an injective map $\alpha^{deg}\otimes\LL_{\chi}[1]\to\pi_r^\star\PPP$ for some $\ell$-adic unit $\alpha$ and, by adjunction, a non-zero map $\alpha^{deg}\otimes\pi_{r\star}\LL_{\chi}[1]\to\PPP$. Since $\PPP$ is irreducible, this map is surjective. In particular $\HHH^0(\PPP)=0$ and $\PPP$ is of the form $\FF[1]$ for some sheaf $\FF$ without punctual sections.

Let $\sigma\in\Gal(k_r/k)$, then $\alpha^{deg}\otimes\LL_{\chi\circ\sigma}[1]=\sigma^\star(\alpha^{deg}\otimes\LL_\chi)[1]\hookrightarrow \sigma^\star\pi_r^\star\PPP\cong\pi_r^\star\PPP$. Moreover $\chi\circ\sigma\neq\chi$ if $\sigma\neq Id$ since otherwise $\chi$ would be the pull-back of a character of $G(k')$, where $k'$ is the subfield of $k_r$ fixed by $\sigma$, and $\LL_\chi$ would be defined over $k'$, contrary to the hypothesis. Therefore $\pi^\star_r\FF$ contains at least $r$ non-isomorphic smooth subsheaves of rank $1$, so its rank (which is the rank of $\FF$) must be at least $r$. We conclude that the map $\alpha^{deg}\otimes\pi_{r\star}\LL_{\chi}[1]\to\PPP$ is an isomorphism. In particular, $\PPP\otimes\bar k$ is the direct sum of $\LL_{\chi\circ\sigma}[1]$ for all $\sigma\in\Gal(k_r/k)$.
\end{proof}

Following \cite{katz2010mellin} we will say that an irreducible perverse sheaf $\PPP$ on $G$ is \emph{negligible} if it satisfies the equivalent conditions in the previous lemma. 

\begin{prop}
 Let $\PPP$ be an irreducible perverse sheaf on $G$ of weight $1$, and suppose that $\PPP$ is non-negligible. Then for every $r\geq 1$, $\PPP^{[\ast r]}\in K_0(G,\QQ)$ is an integral combination of classes of perverse sheaves of weights $\leq r$.   
\end{prop}

\begin{proof}
 By \cite[2.6.4, 2.6.8, 2.6.13, 2.6.14]{katz1996rls} the $r$-fold convolution $\PPP^{\ast r}$ is a perverse sheaf, of weights $\leq r$ (since $K$ is perverse if and only if $K\otimes \bar k$ is). For every representation $\rho$ of ${\mathfrak S}_r$, $\R(\ast\rho)\PPP=(\PPP^{\ast r})^\rho$ is then also perverse of weights $\leq r$ (since $\HHH^i((\PPP^{\ast r})^\rho)=\HHH^i(\PPP^{\ast r})^\rho$ is a subsheaf of $\HHH^i(\PPP^{\ast r})$ for every $i$, and is therefore zero for $i\neq 0,-1$, punctual for $i=0$ and without punctual sections for $i=-1$). In particular $\PPP^{[\ast r]}=\sum_{i=0}^{r-1}(-1)^i\R(\ast\wedge^i\rho)\PPP$ is an integral combination of classes of such perverse sheaves, where $\rho$ is the standard representation of ${\mathfrak S}_r$.
\end{proof}

\begin{cor}\label{UC}
 Let $\PPP$ be a non-negligible irreducible perverse sheaf on $G$ of weight $1$. For every integer $r\geq 1$ there exists a dense open set $U_{\PPP,r}\subseteq G$ and a constant $C_{\PPP,r}$ such that for every integer $m\geq 1$ and every $t\in U_{\PPP,r}(k_m)$ the $L$-function $L^{\Nm,r}({\mathcal P},k_m,t;T)$ has total degree bounded by $C_{\PPP,r}$ and all its reciprocal roots and poles have $q^m$-weight $\leq r-1$. In particular, for every $m\geq 1$ and every $t\in U_{\PPP,r}(k_m)$ we have the estimate
$$
|f^{\Nm,r}_\PPP(k_m,t)|\leq C_{\PPP,r}q^{\frac{m(r-1)}{2}}
$$ 
If $t\notin U_{\PPP,r}(k_m)$, then all reciprocal roots and poles of $L^{\Nm,r}({\mathcal P},k_m,t;T)$ have $q^m$-weight $\leq r$.
\end{cor}

\begin{proof} For every $i=0,\ldots,r-1$, let ${\mathcal Q}_i$ be the perverse sheaf $\R(\ast\wedge^i\rho)\PPP$, where $\rho$ is the standard representation of ${\mathfrak S}_r$. Let $U_i\subseteq G$ be the largest open set on which $\HHH^0({\mathcal Q}_i)=0$, and $C_i$ the generic rank of $\HHH^{-1}({\mathcal Q}_i)$. We define $U_{\PPP,r}=U_0\cap\cdots\cap U_{r-1}$ and $C_{\PPP,r}=C_0+\cdots+C_{r-1}$.

For every integer $m\geq 1$ and every $t\in U_{\PPP,r}(k_m)$, we have
$$
L^{\Nm,r}({\mathcal P},k_m,t;T)=L(\PPP^{[\ast r]},k_m,t;T)=\prod_{i=0}^{r-1}L({\mathcal Q}_i,k_m,t;T)^{(-1)^i}
$$
$$
=\prod_{i=0}^{r-1}L(\HHH^{-1}({\mathcal Q}_i),k_m,t;T)^{(-1)^{i+1}}
$$

The result follows from the fact that $\HHH^{-1}({\mathcal Q}_i)$ is mixed of weights $\leq r-1$ and does not have punctual sections for any $i$, so its rank at any point of $U_{\PPP,r}$ is less than or equal to its generic rank. If $t\notin U_{\PPP,r}(k_m)$ we would also get factors of the form $L(\HHH^{0}({\mathcal Q}_i),k_m,t;T)$ which are mixed of weight $\leq r$.
\end{proof}

The result extends to any perverse $\PPP$ pure of weight $1$ as long as no irreducible component of $\PPP$ is negligible. As an easy example, we compute $U_{\PPP,r}$ and $C_{\PPP,r}$ when $\PPP=\delta_a$ is a punctual object supported on $a\in G(k)$ (placed in degree $0$).

\begin{prop}
 For every $a\in G(k)$ and $r\geq 1$, $\delta_a^{[\ast r]}=\delta_{a^r}$. In particular, $U_{\PPP,r}(\bar k)=G(\bar k)-\{a\}$ and $C_{\PPP,r}=0$.
\end{prop}

\begin{proof}
 For every $r\geq 1$, $\delta_a^{\ast r}=\delta_{a^r}$, with trivial ${\mathfrak S}_r$ action \cite[2.5.3]{katz1996rls}. Therefore $\R(\ast\wedge^i\rho)\delta_a=\delta_{a^r}$ for $i=0$ and $0$ for $i>0$. We conclude that $\delta_a^{[\ast r]}=[\delta_{a^r}]$.
\end{proof}

\begin{cor}
 Let $\PPP$ be a punctual perverse sheaf on $G$ supported on $Z\subseteq G$. Let $S=\{z^r|z\in Z(\bar k)\}$. Then in corollary \ref{UC} one can take $U_{\PPP,r}(\bar k)=G(\bar k)-S$ and $C_{\PPP,r}=0$.
\end{cor}

\begin{proof}
 It is an immediate consequence of the previous proposition and the additivity of the convolution Adams power, since the operation $\PPP\mapsto\PPP^{[\ast r]}$ commutes with extension of scalars to $\bar k$.
\end{proof}

\begin{rem}{\rm 
 This shows that, in general, the $U_{\PPP,r}$ and $C_{\PPP,r}$ defined in the proof of corollary \ref{UC} are not the best possible ones, since there may be some cancellation among the ${\mathcal Q}_i$'s when taking the alternating product. For instance, if $k={\mathbb F}_3$, $G=\AAA^1_k$ and $\PPP$ is punctual supported on $\Spec k[t]/(t^2+1)=\{\pm i\}$, the proof gives $U_{\PPP,2}=G-\{0,\pm 2i\}=G-\{0,\pm i\}$ (since $\PPP^{\ast 2}=\Sym^{\ast 2}\PPP\oplus\wedge^{\ast 2}\PPP$, so the union of the supports of $\Sym^{\ast 2}\PPP$ and $\wedge^{\ast 2}\PPP$ is the support of $\PPP^{\ast 2}$), while the previous corollary shows that we could take $U_{\PPP,2}=G-\{\pm i\}$. 

In the remainder of the article we will assume, unless otherwise stated, that $U_{\PPP;r}$ and $C_{\PPP,r}$ are the ones defined in the proof of \ref{UC}.}
\end{rem}

\section{Examples on $\AAA^1_k$}

In this section we further specialize to the case $G=\AAA^1_k$. Here the ``norm'' map $G(k_{mr})=k_{mr}\to G(k_m)=k_m$ is just the trace, so we will write $L^{\Tr,r}$ and $f^{\Tr,r}$ instead of $L^{\Nm,r}$ and $f^{\Nm,r}$. Fix a non-trivial character $\psi:k\to\QQ^\star$. Since the Fourier transform with respect to $\psi$ preserves perversity and interchanges punctual objects and (shifted) Artin-Schreier sheaves, an irreducible perverse object $\PPP$ on $\AAA^1_k$ is negligible if and only if its Fourier transform is punctual. For those objects we can explicitely determine $\PPP^{[\ast r]}$:

\begin{prop}\label{ASformula}
 Let ${\mathcal P}\in\Perv(\AAA^1_k,\QQ)$ be irreducible, negligible and pure of weight $1$. Then
$$
\PPP^{[\ast r]}=[(\alpha q)^{(r-1) deg}\otimes\PPP]
$$
for some $\ell$-adic unit $\alpha\in\QQ$ of weight $0$.
\end{prop}

\begin{proof} After taking Fourier transform with respect to $\psi$ on both sides, the equality is equivalent by corollary \ref{fourieradams} to
$$
{\mathcal Q}^{[r]}=[(\alpha q)^{(r-1) deg}\otimes{\mathcal Q}]
$$
where $\mathcal Q$ is the Fourier transform of $\PPP$, which is punctual and pure of weight $2$ by hypothesis.

Since ${\mathcal Q}$ is punctual and irreducible, there exists a closed point $x\in\AAA^1_k$ and an irreducible sheaf $\FF$ on $\{x\}=\Spec k(x)$ such that ${\mathcal Q}=i_{x\star}\FF[0]$. Since $\FF$ is irreducible, it must be equal to $(\alpha q)^\mathrm{deg}$ for some $\ell$-adic unit $\alpha$ of weight $0$. We will show that ${\mathcal Q}^{[r]}=[(\alpha q)^{(r-1) deg}\otimes{\mathcal Q}]$ by comparing their trace functions and using the injectivity of $\Phi$.

For any $m\geq 1$ and every $t\in k_m$, $\Phi({\mathcal Q}^{[r]})(k_m,t)=\Tr(\Frob_{k_m,t}|i_{x\star}\FF_{\bar t}^{[r]})=\Tr(\Frob_{k_{mr},t}|i_{x\star}\FF_{\bar t})=(\alpha q)^{mr}$ if $t\in\{x\}(k_{mr})$ and $0$ otherwise. But $t\in\{x\}(k_{mr})$ if and only if $t\in\{x\}(k_m)$ (if and only if $t$ is a root of the irreducible polynomial that defines $x$), so in any case $\Phi({\mathcal Q}^{[r]})(k_m,t)=(\alpha q)^{m(r-1)}\Phi({\mathcal Q})(k_m,t)=\Phi((\alpha q)^{(r-1) deg}\otimes{\mathcal Q})(k_m,t)$.
\end{proof}

The $\ell$-adic unit $\alpha$ can be determined from ${\mathcal P}$ in the following way: Let $d\geq 1$ be an integer such that $\HHH^{-1}({\mathcal P})\otimes k_d$ splits as an extension of Artin-Schreier sheaves. Then $\alpha$ is any $d$-th root of
$$
\frac{\Tr(\Frob_{k_d,0}|(\HHH^{-1}({\mathcal P})\otimes k_d)_0)}{\mathrm{rank}\,\HHH^{-1}({\mathcal P})}.
$$ 

\begin{cor}\label{AScase}
 If ${\mathcal P}\in\Perv(\AAA^1_k,\QQ)$ is irreducible, negligible and pure of weight $1$, then there exists some $\ell$-adic unit $\alpha\in\QQ$ of weight $0$ such that for every integer $m\geq 1$ and every $t\in k_m$ 
$$
L^{\Tr,r}({\mathcal P},k_m,t;T)=L({\mathcal P},k_m,t;(\alpha q)^{m(r-1)}T).
$$
\end{cor}

For non-negligible $\PPP$ we can give the following characterization of $U_{\PPP,r}$:

\begin{prop}\label{ULr}
 Let $\PPP\in\Perv(\AAA^1_k,\QQ)$ be irreducible and non-negligible. Let $\mathrm{FT}_\psi\PPP=\GGG[1]$, with $\GGG\in\Sh(\AAA^1_k,\QQ)$ an irreducible middle extension sheaf. Let $a\in \bar k$, and $\LL_{\psi_{-a}}$ the Artin-Schreier sheaf on $\AAA^1_{\bar k}$ associated to the character $t\mapsto\psi(-at)$. Then $a\in U_{\PPP,r}(\bar k)$ if and only if the following equivalent conditions hold:
\begin{enumerate}
 \item $\HH^2_c(\AAA^1_{\bar k},\LL_{\psi_{-a}}\otimes\R(\wedge^i\rho)\GGG)=0$ for every $i=0,\ldots,r-1$, where $\rho$ is the standard representation of ${\mathfrak S}_r$.
 \item $\HH^2_c(\AAA^1_{\bar k},\LL_{\psi_{-a}}\otimes\Sym^{r-i}\GGG\otimes\wedge^i\GGG)=0$ for every $i=0,\ldots,r$.
\end{enumerate}
\end{prop}

\begin{proof}
 For (1) it is a consequence of the definition of $U_{\PPP,r}$ in proposition \ref{UC}: since $\GGG[1]$ is the Fourier transform of $\PPP$, by \ref{fourier} and \ref{shift} the Fourier transform of $\R(\ast\wedge^i\rho)\PPP$ is $\R((\wedge^i\rho)\otimes\sigma)(\GGG[1])[1-r]\cong(\R(\wedge^i\rho)\GGG[r])[1-r]=\R(\wedge^i\rho)\GGG[1]$, so  $\HH^2_c(\AAA^1_{\bar k},\LL_{\psi_{-a}}\otimes\R(\wedge^i\rho)\GGG)$ is (a Tate twist of) the stalk at $a$ of $\HHH^0({\mathcal Q}_i)$.

The equivalence between (1) and (2) can be deduced from the formulas $\Sym^{r-i}\GGG\otimes\wedge^i\GGG=\R(\wedge^{i-1}\rho\oplus\wedge^i\rho)\GGG=\R(\wedge^{i-1}\rho)\GGG\oplus\R(\wedge^i\rho)\GGG$ for $i=1,\ldots,r-1$, $\Sym^r\GGG=\R({\mathbf 1})\GGG$ and $\wedge^r\GGG=\R(\wedge^{r-1}\rho)\GGG$.
\end{proof}

We will now apply these results to some particular examples of sheaves.

\begin{prop}\label{trex1}
 Suppose that $\PPP=\FF[1]$ where $\FF\in{\mathcal Sh}(\AAA^1_{\bar k},\QQ)$ is a geometrically semisimple middle extension sheaf of generic rank $d$, Euler characteristic $-e$ and Swan conductor at infinity $c$ such that $\HH^2_c(\AAA^1_{\bar k},\FF)=0$ and all its slopes at infinity are $<1$ (e.g. $\FF$ tamely ramified at infinity). Let $S:=\{a_1,\ldots,a_s\}\subset \bar k$ be the set of ramification points of $\FF\otimes\bar k$, and $S_r:=S+\cdots +S$ ($r$ summands). Then for every $r\geq 1$ $U_{\PPP,r}(\bar k)$ contains $\bar k-S_r$, and $C_{\PPP,r}$ is bounded by
$$(1+c)\sum_{i=0}^{r-1} \left[{{d+e-c+r-i-1}\choose{r}}-{{e+r-i-1}\choose{r}}\right]{r-1\choose i}.
$$
\end{prop}

\begin{proof}
 First of all, $\PPP$ does not have negligible subquotients: since $\HH^2_c(\AAA^1_{\bar k},\FF)=0$ $\PPP\otimes\bar k$ can not have constant subsheaves, and it can not have non-trivial Artin-Schreier subsheaves either since $1$ is not a slope at infinity.

Therefore $\mathrm{FT}_\psi\PPP=\GGG[1]$, where $\GGG\in\Sh(\AAA^1_k,\QQ)$ is a middle extension sheaf. Laumon's local Fourier transform theory \cite[2.4]{laumon1987transformation} implies that $\GGG$ is smooth on $\GG_{m,\bar k}$, since $1$ is not a slope of $\FF$ at infinity. Its generic rank is $\dim\HH^1_c(\AAA^1_{\bar k},\FF\otimes\LL_{\psi_a})$ for any $a\neq 0$, that is, $\dim\HH^1_c(\AAA^1_{\bar k},\FF)-\Swan_\infty\FF+\Swan_\infty(\FF\otimes\LL_{\psi_a})=e-c+d$ (since all slopes at infinity of $\FF\otimes\LL_{\psi_a}$ are equal to $1$). Its rank at $0$ is $\dim\HH^1_c(\AAA^1_{\bar k},\FF)=e$. At infinity, it is the direct sum, for $s\in S$, of ${\FF}_s\otimes\LL_{\psi_s}$, where ${\FF}_s$ is the local Fourier transform operator $\FF^{(0,\infty)}$ \cite[2.4.2.3]{laumon1987transformation} applied to the local monodromy of $\FF$ at $s$. Its Swan conductor at $0$ is $c$, since for generic $a$ the dimension of $\HH^1_c(\AAA^1_{\bar k},\GGG\otimes\LL_{\psi_a})$ is $d$ (by the involutivity of Fourier transform) and, by the Euler-Poincar\'e formula \cite[Expos\'e X, Corollaire 7.12]{grothendieck1977cohomologie}, it is also equal to $\Swan_\infty(\GGG\otimes\LL_{\psi_a})+\Swan_0(\GGG\otimes\LL_{\psi_a})-\dim\GGG_0=(e-c+d)+\Swan_0\,\GGG-e$.

Then for every $i=0,\ldots,r-1$ the sheaf $\R(\wedge^i\rho)\GGG$ is smooth on $\GG_{m,\bar k}$, has generic rank ${{d+e-c+r-i-1}\choose{r}}{r-1\choose i}$ (cf. remark \ref{rank}), its rank at $0$ is ${{e+r-i-1}\choose{r}}{r-1\choose i}$ and its monodromy action at infinity splits as a direct sum ${\mathcal N}_s\otimes\LL_{\psi_s}$ for $s\in S_r$, where ${\mathcal N}_s$ is a representation of $I_\infty$ whose slopes are all $<1$. In particular, for every $a\in\bar k-S_r$, the sheaf $\LL_{\psi_{-a}}\otimes\R(\wedge^i\rho)\GGG$ is totally wild at infinity, so $a\in U_{\LL,r}(\bar k)$ by proposition \ref{ULr}.

Furthermore, the dimension of $\HH^1_c(\AAA^1_c,\LL_{\psi_{-a}}\otimes\R(\wedge^i\rho)\GGG)$ is, by the Euler-Poincar\'e formula, equal to $$\Swan_0(\LL_{\psi_{-a}}\otimes\R(\wedge^i\rho)\GGG)+$$
$$+\Swan_\infty(\LL_{\psi_{-a}}\otimes\R(\wedge^i\rho)\GGG)-{{e+r-i-1}\choose{r}}{r-1\choose i}\leq
$$
$$
\leq c\left[{{d+e-c+r-i-1}\choose{r}}{r-1\choose i}-{{e+r-i-1}\choose{r}}{r-1\choose i}\right]+
$$
$$+{{d+e-c+r-i-1}\choose{r}}{r-1\choose i}-{{e+r-i-1}\choose{r}}{r-1\choose i}\leq$$
$$
\leq (1+c)\left[{{d+e-c+r-i-1}\choose{r}}-{{e+r-i-1}\choose{r}}\right]{r-1\choose i}.
$$
since the Swan conductor of $\GGG$ at $0$ (and therefore all its slopes) are less than or equal to $c$. We conclude by applying the formula for $C_{\PPP,r}$ in the proof of proposition \ref{UC}.
\end{proof}

Our first example improves \cite[Corollary 4]{rlwan2010}:

\begin{exa}
 Let $g\in k[x]$ be a polynomial of degree $d$ prime to $p$ and $\PPP=\FF[1]$, where $\FF$ is the kernel of the trace map $g_\star\QQ\to\QQ$. Let $S$ be the set of critical values of $g$, and $S_r=S+\cdots+S$ ($r$ summands). Then $\bar k-S_r\subseteq U_{\PPP,r}(\bar k)$ for every $r\geq 1$. Therefore, for every $m\geq 1$ and every $t\in k_m$ which is not the sum of $r$ critical values of $g$ (i.e. such that the affine hypersurface $g(x_1)+\cdots+g(x_r)=t$ in $\AAA^r_{\bar k}$ is smooth) we have
$$
\left| \#\{x\in k_{mr}|\Tr_{k_{mr}/k_m}(g(x))=t\}-q^{m(r-1)}\right|
\leq\sum_{i=0}^{r-1} {{d+r-i-2}\choose{r}}{{r-1}\choose i} q^{\frac{m(r-1)}{2}}.
$$

In particular, if the affine hypersurface $g(x_1)+\cdots+g(x_r)=0$ in $\AAA^r_{\bar k}$ is smooth we have
$$
\left|\#\{(x,y)\in k_{mr}^2|y^{q^m}-y=g(x)\}-q^{mr}\right|
\leq \sum_{i=0}^{r-1} {{d+r-i-2}\choose{r}}{{r-1}\choose i} q^{\frac{m(r+1)}{2}}.
$$
\end{exa}

\begin{proof}
 The left hand side is $|f^{\Tr,r}_{g_\star\QQ}(k_m,t)-f^{\Tr,r}_{\QQ}(k_m,t)|=|f^{\Tr,r}_{\FF}(k_m,t)|$. By Proposition \ref{ASformula} and the comment after it, $f^{\Tr,r}_{\QQ}=\kappa_{q^{r-1}}\cdot f_{\QQ}$, so $f^{\Tr,r}_{\QQ}(k_m,t)=q^{m(r-1)}$.

On the other hand, $\FF$ has rank $d-1$ and satisfies the hypotheses of proposition \ref{trex1} with $e=c=0$, since it is tamely and totally ramified at infinity (the inertia group at infinity acts via the direct sum of all its non-trivial characters with trivial $d$-th power) and there is an exact sequence $0\to\FF\to g_\star\QQ\to\QQ\to0$ with $\dim\HH^1_c(\AAA^1_{\bar k},g_\star\QQ)=\dim\HH^1_c(\AAA^1_{\bar k},\QQ)=0$. The first inequality follows. The second one is an easy consequence of the identity
$$
\#\{(x,y)\in k_{mr}^2|y^{q^m}-y=g(x)\}=q^m\cdot\#\{x\in k_{mr}|\Tr_{k_{mr}/k_m}(g(x))=0\}.
$$
\end{proof}

\begin{exa}
 Let $\chi:k^\star\to\QQ^\star$ be a non-trivial multiplicative character of order $n$, $g\in k[x]$ a non-constant polynomial which has no roots in $\bar k$ with multiplicity divisible by $n$, $\FF:=\LL_{\chi(g)}=g^\star\LL_\chi$ and $\PPP=\FF[1]$. Let $S\subseteq \bar k$ be the set of roots of $g$, and $S_r=S+\cdots+S$ ($r$ summands). Then $\bar k-S_r\subseteq U_{\PPP,r}(\bar k)$ for every $r\geq 1$. In particular, for every $m\geq 1$ and every $t\in k_m$ which is not the sum of $r$ roots of $g$,
$$
\left| \sum_{\Tr_{k_{mr}/k_m}(x)=t}\chi(\Nm_{k_{mr}/k}(g(x)))\right|
\leq\sum_{i=0}^{r-1}{{a+r-i-2}\choose{r-1}}{{r-1}\choose i} q^{\frac{m(r-1)}{2}}.
$$
where $a$ is the number of distinct roots of $g$ in $\bar k$.
\end{exa}

\begin{proof}
 The sheaf $\FF$ is a middle extension of rank $1$ ramified at the roots of $g$ and therefore it is not isomorphic to an Artin-Schreier sheaf. We can then apply Proposition \ref{trex1} to it, where $d=1$, $e=a-1$ by the Euler-Poincar\'e formula and $c=0$ (since the inertia group at infinity acts on $\FF$ via a power of the tame character $\chi$).
\end{proof}

\begin{prop}\label{trex2}
 Under the hypotheses of proposition \ref{trex1}, suppose further that the action of the inertia group at every point of $\AAA^1_{\bar k}$ on the generic stalk of $\FF$ (modulo its invariant subspace under this action) is a successive extension of a fixed tame character $\chi$ of order $n$. Then $U_{\PPP,r}=\AAA^1_k$ for every $r\geq 1$ which is not divisible by $n$. 
\end{prop}

\begin{proof}
 In this case, since the local Fourier transform of a tame character is its conjugate, the representations ${\mathcal F}_s$ in the proof of proposition \ref{trex1} are successive extensions of the character $\bar\chi$, so the representations ${\mathcal N}_s$ appearing (tensored with Artin-Schreier characters) in the monodromy at infinity of $\R(\wedge^i\rho)\GGG$ are successive extensions of the character $\bar\chi^r$. In particular, the action of the inertia group $I_\infty$ on $\R(\wedge^i\rho)\GGG$ has no invariants if $\chi^r$ is non-trivial (since the tensor product of a non-trivial tame character and a (possibly trivial) Artin-Schreier character can not be trivial). 
\end{proof}

\begin{exa}
 Let $g\in k[x]$ be a polynomial of degree $d$ prime to $p$, and let $\FF$ be the kernel of the trace map $g_\star\QQ\to\QQ$. Suppose that $p>2$ and the derivative $g'$ has no multiple roots. Then for every odd $r\geq 1$, every $m\geq 1$ and every $t\in k_m$ we have
$$
\left| \#\{x\in k_{mr}|\Tr_{k_{mr}/k_m}(g(x))=t\}-q^{m(r-1)}\right|
\leq\sum_{i=0}^{r-1} {{d+r-i-2}\choose{r}}{{r-1}\choose i} q^{\frac{m(r-1)}{2}}
$$
and, in particular,
$$
\left|\#\{(x,y)\in k_{mr}^2|y^{q^m}-y=g(x)\}-q^{mr}\right|
\leq \sum_{i=0}^{r-1} {{d+r-i-2}\choose{r}}{{r-1}\choose i} q^{\frac{m(r+1)}{2}}.
$$
\end{exa}

\begin{proof}
 In this case the monodromy of $\FF$ at each ramified finite point is a successive extension of the quadratic character by the hypothesis on $g'$.
\end{proof}

\begin{exa}
 Let $\chi:k^\star\to\QQ^\star$ be a non-trivial multiplicative character of order $n$, $g\in k[x]$ a square-free polynomial of degree $d$ and $\FF:=\LL_{\chi(g)}=g^\star\LL_\chi$. Then for every $r\geq 1$ not divisible by $n$, every $m\geq 1$ and every $t\in k_m$ we have
$$
\left| \sum_{\Tr_{k_{mr}/k_m}(x)=t}\chi(\Nm_{k_{mr}/k}(g(x)))\right|
\leq\sum_{i=0}^{r-1}{{d+r-i-2}\choose{r-1}}{{r-1}\choose i} q^{\frac{m(r-1)}{2}}.
$$
\end{exa}

\begin{proof}
 In this case the inertia groups at all ramified finite points act on $\FF$ via $\chi$. Additionally, since $g$ is square-free, it has $d$ distinct roots on $\bar k$.
\end{proof}

In order to get more precise results, we need to consider the global monodromy. Suppose that $\PPP\in\Perv(\AAA^1_k,\QQ)$ is pure of weight $1$ and does not have any negligible subquotient. Then $\FT_\psi\PPP=\GGG[1]$, where $\GGG\in\Sh(\AAA^1_k,\QQ)$ is a middle extension sheaf, pure of weight $1$ (as a middle extension). Let $V$ be the generic stalk of $\GGG$, $G\subseteq\mathrm{GL}(V)$ its global geometric monodromy group and $G_0\subseteq G$ its unit connected component. Since $\GGG$ is pure, its restriction to any open set on which it is smooth is geometrically semisimple \cite[Th\'eor\`eme 3.4.1]{deligne1980conjecture} and therefore $G_0$ is a semisimple algebraic group \cite[Corollaire 1.3.9]{deligne1980conjecture}. 

\begin{prop}
 Under the previous hypotheses, $0\in U_{\PPP,r}$ if and only if for every $i=0,\ldots,r$ the representation $\Sym^{r-i} V\otimes\wedge^i V$ of $G$ has no non-zero invariants.
\end{prop}

\begin{proof}
 This is a restatement of proposition \ref{ULr}, since $\HH^2_c(\AAA^1_{\bar k},\Sym^{r-i}\GGG\otimes\wedge^i\GGG)$ is the coinvariant space of $\Sym^{r-i} V\otimes\wedge^i V$ under the action of $\pi_1(U)$ (where $U$ is the largest open subset of $\AAA^1_k$ on which $\GGG$ is smooth), which has the same dimension as the invariant space.
\end{proof}

\begin{prop}\label{gm}
 Under the previous hypotheses, suppose that $G/G_0$ has order prime to $p$. Then $\GG_{m,k}\subseteq U_{\PPP,r}$ for every $r\geq 1$.
\end{prop}

\begin{proof}
 Otherwise there would exist some $a\in\bar k^\star$ and some $i=0,\ldots,r$ such that the representation $\LL_{\psi_{-a}}\otimes\Sym^{r-i} V\otimes\wedge^i V$ of $G$ has non-zero invariants. In other words, the representation $\Sym^{r-i} V\otimes\wedge^i V$ contains a subcharacter $\LL_{\psi_a}$ of order $p$. Then its kernel $G'$ would be a closed normal subgroup of $G$ of index $p$, so $G/G_0$ would contain the $p$-group $G/G'$.
\end{proof}

\begin{exa}
 Let $g\in k[x]$ be a polynomial of degree $d$ prime to $p$ and $\PPP=\FF[1]$ where $\FF$ is the kernel of the trace map $g_\star\QQ\to\QQ$. Suppose that $p>2d-1$ and the $(d-1)(d-2)$ differences between pairs of critical values of $g$ are all distinct. Let $s$ be the sum of the $d-1$ critical values of $g$. Then 
\begin{enumerate}
 \item $\AAA^1_k-\{\frac{rs}{d-1}\}\subseteq U_{\PPP,r}$ for any $r$.
\item $\frac{rs}{d-1}\in U_{\PPP,r}$ for any $r$ if $d$ is even, and for any $r\neq d-1$ if $d$ is odd.
\end{enumerate}
\end{exa}

\begin{proof}
 By \cite[Theorem 7.9.6]{katz1990esa}, in this case $G_0=\mathrm{SL}(V)$, and by \cite[Proposition 4.1]{rlwan2010}, $G/G_0$ has order $1,2,p$ or $2p$ in the cases $s=0, d\text{ odd}; s=0, d\text{ even}; s\neq 0, d\text{ odd and }s\neq 0, d\text{ even}$. For $s=0$, (1) is a consequence of Proposition \ref{gm} and (2) is proven in \cite[Corollary 4.2]{rlwan2010}.

 Suppose that $s\neq 0$, and let $h(x)=g(x)-s/(d-1)$ be the translation of $g$ by $-s/(d-1)$. Then the critical values of $h$ add up to $0$. Let $\FF'=\ker\,(\mathrm{tr}:h_\star\QQ\to\QQ)=\tau_{-s/(d-1)}\FF$, then its Fourier transform $\GGG'$ is $\GGG\otimes\LL_{\psi_{-s/(d-1)}}$, so $\GGG=\GGG'\otimes\LL_{\psi_{s/(d-1)}}$. Therefore $\LL_{\psi_{-a}}\otimes\Sym^{r-i}\GGG\otimes\wedge^i\GGG=\LL_{\psi_{rs/(d-1)-a}}\otimes\Sym^{r-i}\GGG'\otimes\wedge^i\GGG'$. As seen above, the monodromy group of $\GGG'$ is either $\mathrm{SL}(V)$ or $\{\pm 1\}\times\mathrm{SL}(V)$ and does not have characters of order $p$, so $\LL_{\psi_{rs/(d-1)-a}}\otimes\Sym^{r-i}\GGG'\otimes\wedge^i\GGG'$ can only have non-zero invariants for $a=rs/(d-1)$, and in that case only for $r=d-1$ and $d$ odd \cite[Corollary 4.2]{rlwan2010}.
\end{proof}

\begin{exa}
 Let $g\in k[x]$ be a polynomial of degree $d\geq 3$ prime to $p$, $\psi:k\to\QQ^\star$ a non-trivial additive character and $\PPP=\LL_{\psi(g)}[1]$. Suppose that $p>2d+1$ and $g(x+a)+b$ is not odd for any $a,b\in\bar k$. Let $a_{d-1}$ be the coefficient of $x^{d-1}$ in $g$. Then 
\begin{enumerate}
\item $\AAA^1_k-\{\frac{ra_{d-1}}{d}\}\subseteq U_{\PPP,r}$ for any $r\geq 1$.
 \item $\frac{ra_{d-1}}{d}\in U_{\PPP,r}$ for any $r\neq d-1$.
\end{enumerate}
In all such cases, for $t\in U_{\PPP,r}(k_m)$,
$$
\left| \sum_{\Tr_{k_{mr}/k_m}(x)=t}\psi(\Tr_{k_{mr}/k}(g(x)))\right|\leq\frac{1}{d-1}\sum_{i=0}^{r-1} {{d+r-i-2}\choose{r}}{{r-1}\choose i} q^{\frac{m(r-1)}{2}}.
$$
\end{exa}

\begin{proof} The Swan conductor of $\LL_{\psi(g)}$ at infinity is $d>1$, so $\PPP$ is not geometrically an Artin-Schreier object. Let $h(x)=g(x-\frac{a_{d-1}}{d})$. Then the coefficient of $x^{d-1}$ in $h$ is $0$. Let $\PPP'=\LL_{\psi(h)}[1]=\tau_{-a_{d-1}/d}\PPP$, $\GGG[1]$ and $\GGG'[1]$ the Fourier transforms of $\PPP$ and $\PPP'$ respectively. Then $\GGG'=\GGG\otimes\LL_{\psi_{-a_{d-1}/d}}$ and $\LL_{\psi_{-a}}\otimes\Sym^{r-i}\GGG\otimes\wedge^i\GGG=\LL_{\psi_{ra_{d-1}/d-a}}\otimes\Sym^{r-i}\GGG'\otimes\wedge^i\GGG'$. By \cite[Theorem 19]{katz-monodromy}, the monodromy group of $\GGG'$ is $G=\mathrm{SL}(V)$. Since $\Sym^{r-i}V\otimes\wedge^i V=\Hom(\wedge^{d-1-i} V,\Sym^{r-i} V)$ only has $\mathrm{SL}(V)$-invariants for $r=d-1$ and $G$ does not have characters of order $p$, we conclude that $\HH^2_c(\AAA^1_{\bar k},\LL_{\psi_{-a}}\otimes\Sym^{r-i}\GGG\otimes\wedge^i\GGG)$ vanishes as long as $a\neq \frac{ra_{d-1}}{d}$ or $r\neq d-1$. 

In that case, since $\R(\wedge^i\rho)\GGG$ is smooth on $\AAA^1_k$ and all its slopes at infinity are $\leq \frac{d}{d-1}$ (because all slopes of $\GGG$ at infinity are equal to $\frac{d}{d-1}$ by \cite[Theorem 7.5.4]{katz1990esa}), by the Euler-Poincar\'e formula $\dim\,\HH^1_c(\AAA^1_{\bar k},\LL_{\psi_a}\otimes\R(\wedge^i\rho)\GGG)=\Swan_\infty(\LL_{\psi_a}\otimes\R(\wedge^i\rho)\GGG)-\mathrm{rank}(\LL_{\psi_a}\otimes\R(\wedge^i\rho)\GGG)\leq (\frac{d}{d-1}-1){{d+r-i-2}\choose{r}}{{r-1}\choose i}$ (cf. remark \ref{rank}), so
$$
C_{\PPP,r}\leq\frac{1}{d-1}\sum_{i=0}^{r-1} {{d+r-i-2}\choose{r}}{{r-1}\choose i}.
$$
\end{proof}

\begin{exa}
 Let $g\in k[x]$ be a polynomial of degree $d\geq 3$ prime to $p$, $\psi:k\to\QQ^\star$ a non-trivial additive character and $\PPP=\LL_{\psi(g)}[1]$. Suppose that $p>2d+1$ and $g(x+a)+b$ is odd for some $a,b\in\bar k$ (so $d$ is necessarily odd). Let $a_{d-1}$ be the coefficient of $x^{d-1}$ in $g$. Then 
\begin{enumerate}
\item $\AAA^1_k-\{\frac{ra_{d-1}}{d}\}\subseteq U_{\PPP,r}$ for any $r\geq 1$.
 \item $\frac{ra_{d-1}}{d}\in U_{\PPP,r}$ for any $r\neq 2t$ for $t=1,\ldots,\frac{d-1}{2}$.
\end{enumerate}
In all such cases, for $t\in U_{\PPP,r}(k_m)$,
$$
\left| \sum_{\Tr_{k_{mr}/k_m}(x)=t}\psi(\Tr_{k_{mr}/k}(g(x)))\right|\leq\frac{1}{d-1}\sum_{i=0}^{r-1} {{d+r-i-2}\choose{r}}{{r-1}\choose i} q^{\frac{m(r-1)}{2}}.
$$
\end{exa}

\begin{proof}
 The proof is similar to the previous one. In this case, $G=\mathrm{Sp}(V)$ by \cite[Theorem 19]{katz-monodromy} so, by \cite[Lemma on p.18]{katz2001frobenius}, $\Sym^{r-i}V\otimes\wedge^i V=\Hom(\wedge^{i} V,\Sym^{r-i} V)$ only has $\mathrm{Sp}(V)$-invariants for even $r\leq d-1$. 
\end{proof}

\section{Examples on $\GG_{m,k}$}

In this section we will assume $G=\GG_{m,k}$. As in the $\AAA^1_k$ case, we will first determine the convolution Adams powers of negligible objects. By lemma \ref{negligible}, such an object is a twist of an object of the form $\pi_{d\star}\LL_\chi[1]$ where $\pi_d:\GG_{m,k_d}\to\GG_{m,k}$ is the projection and $\chi:k_d^\star\to\QQ^\star$ is a character which is not the pullback by the norm map of a multiplicative character of a proper subfield. In other words, $d$ is the smallest positive integer such that $q^d-1$ is a multiple of the order of $\chi$.

\begin{prop}\label{negligiblemult}
 Let $\PPP=\pi_{d\star}\LL_\chi[1]$, where $\chi:k_d^\star\to\QQ^\star$ is a character of order $n\geq 1$, $d$ is the smallest integer such that $n|q^d-1$ and $\pi_d:\GG_{m,k_d}\to \GG_{m,k}$ is the projection. Then
$$
\PPP^{[\ast r]}=\sum_{i=0}^{r-1}[q^{i\cdot deg}\otimes\PPP]
$$
for every $r\geq 1$.
\end{prop}

\begin{proof}
 We will show that both sides have the same trace function, that is, that for every $m\geq 1$ and every $t\in k_m^\star$ we have
$$
f^{\Nm,r}_\PPP(k_m,t)=\sum_{i=0}^{r-1}q^{mi} f_\PPP(k_m,t).
$$

By definition of $\PPP$, we have
$$
f_\PPP(k_m,t)=\left\{\begin{array}{ll}
                     0 & \mbox{if }n\not |q^m-1 \\
                  -\sum_{i=0}^{d-1}\chi^{q^i}(\Nm_{k_m/k_d}(t)) & \mbox{if }n|q^m-1
                    \end{array}
\right.
$$
and
$$
f^{\Nm,r}_\PPP(k_m,t)=\left\{\begin{array}{ll}
                     0 & \mbox{if }n\not |q^{mr}-1 \\
                  -\sum_{\Nm_{k_{mr}/k_m}(u)=t}\sum_{i=0}^{d-1}\chi^{q^i}(\Nm_{k_{mr}/k_d}(u)) & \mbox{if }n|q^{mr}-1
                    \end{array}
\right.
$$

If $n\not |q^{mr}-1$, the equality is obvious. If $n|q^m-1$, the left hand side is
$$
\sum_{\Nm_{k_{mr}/k_m}(u)=t}-\sum_{i=0}^{d-1}\chi^{q^i}(\Nm_{k_{mr}/k_d}(u))=\sum_{\Nm_{k_{mr}/k_m}(u)=t}-\sum_{i=0}^{d-1}\chi^{q^i}(\Nm_{k_m/k_d}(t))=
$$
$$
=\frac{q^{mr}-1}{q^m-1}\cdot f_\PPP(k_m,t)=\left(\sum_{i=0}^{r-1}\kappa_{q^i}(k_m,t)\right)\cdot f_\PPP(k_m,t)
$$
so the equality holds. It remains to prove that $f^{r,\times}_\PPP(k_m,t)=0$ in the case where $n|q^{mr}-1$ but $n\not|q^m-1$.

In that case, we claim that there is an element $u_0\in k_{mr}$ such that $\Nm_{k_{mr}/k_m}(u_0)=1$ but $u_0^{\frac{q^{mr}-1}{n}}\neq 1$. Otherwise, the polynomial $x^{\frac{q^{mr}-1}{q^m-1}}-1$ would divide $x^{\frac{q^{mr}-1}{n}}-1$, so $\frac{q^{mr}-1}{q^m-1}$ would divide $\frac{q^{mr}-1}{n}$, which is impossible since $n$ does not divide $q^m-1$. Then $\Nm_{k_{mr}/k_m}(u)=\Nm_{k_{mr}/k_m}(uu_0)$, so for every $i=0,\ldots,d-1$
$$
\sum_{\Nm_{k_{mr}/k_m}(u)=t}\chi^{q^i}(\Nm_{k_{mr}/k_d}(u)) =\sum_{\Nm_{k_{mr}/k_m}(u)=t}\chi^{q^i}(\Nm_{k_{mr}/k_d}(uu_0))=
$$
$$
=\chi^{q^i}(\Nm_{k_{mr}/k_d}(u_0))\sum_{\Nm_{k_{mr}/k_m}(u)=t}\chi^{q^i}(\Nm_{k_{mr}/k_d}(u)).
$$
Now since the character $\chi^{q^i}\circ\Nm_{k_{mr}/k_d}$ of $k_{mr}^\star$ has order $n$ and $u_0^{\frac{q^{mr}-1}{n}}\neq 1$ it follows that $\chi^{q^i}(\Nm_{k_{mr}/k_d}(u_0))\neq 1$, and $\sum_{\Nm_{k_{mr}/k_m}(u)=t}\chi^{q^i}(\Nm_{k_{mr}/k_d}(u))$ must then be zero. So $f^{\Nm,r}_\PPP(k_m,t)=0$.
\end{proof}

In order to compute explicitely the $r$-th norm $L$-function of a given perverse sheaf, we first split the negligible components from the non-negligible ones. For the negligible components proposition \ref{negligible} gives us an exact formula, so let us focus on the non-negligible objects.

\begin{prop}\label{Sr}
 Let $\FF\in\Shg$ be a geometrically semisimple middle extension sheaf without negligible components which is tamely ramified at both $0$ and $\infty$, let $S\subseteq\bar k^\star$ be the set of finite ramification points of $\FF$ and $S^r:=S\cdots S$ ($r$ factors). Then for every $r\geq 1$, $\bar k-S^r\subseteq U_{\FF[1],r}(\bar k)$.
\end{prop}

\begin{proof}
 By \cite[Lemma 19.5]{katz2010mellin}, $\HHH^0(\FF[1]^{\ast r})$ vanishes on $\GG_{m,k}-S^r$, so the same is true for $\HHH^0(\R(\ast\wedge^i\rho)\FF[1])={\mathcal Hom}_{{\mathfrak S}_r}(\wedge^i\rho,\HHH^0(\FF[1]^{\ast r}))$ for every $i=0,\ldots,r-1$.
\end{proof}

\begin{prop}\label{tame}
 Let $\FF,\GGG\in\Shg$ be geometrically semisimple middle extension sheaves without negligible components which are everywhere tamely ramified. Then $\HHH^{-1}(\FF[1]\ast\GGG[1])$ is everywhere tamely ramified.
\end{prop}

\begin{proof}
 Let $S\subseteq\bar k^\star$ (respectively $T\subseteq \bar k^\star$) be the set of ramification points of $\FF$ (resp. $\GGG$). Let $m$ (resp. $n$) be the generic rank of $\FF$ (resp. $\GGG$), and for every $s\in S$ (resp. $t\in T$) let $m_s$ (resp. $n_t$) be the rank of $\FF$ at $s$ (resp. the rank of $\GGG$ at $t$). By the Euler-Poincar\'e formula,
$$
\chi(\FF[1])=\sum_{s\in S}(m-m_s),
$$
$$
\chi(\GGG[1])=\sum_{t\in T}(n-n_t)
$$
and
$$
\chi(\FF[1]\ast\GGG[1])=\chi(\FF[1])\cdot\chi(\GGG[1])=\left(\sum_{s\in S}(m-m_s)\right)\left(\sum_{t\in T}(n-n_t)\right).
$$

By \cite[Lemma 19.5]{katz2010mellin}, $\FF[1]\ast\GGG[1]$ is smooth on $\bar k^\star-ST$. If $u\in \bar k^\star-ST$, the rank of $\HHH^{-1}(\FF[1]\ast\GGG[1])$ at $u$ is $-\chi(\GG_{m,\bar k},\FF\otimes\phi_u^\star\GGG)$, where $\phi_u:\GG_{m,\bar k}\to\GG_{m,\bar k}$ is the automorphism defined by $t\mapsto u/t$. Since, at every point of $\GG_{m,\bar k}$, at least one of $\FF$, $\phi_u^\star\GGG$ is smooth, and each local term in the Euler-Poincar\'e formula \cite[Expos\'e X, Corollaire 7.12]{grothendieck1977cohomologie} gets multiplied by $d$ upon tensoring with a smooth sheaf of rank $d$, we conclude that 
$$
-\chi(\GG_{m,\bar k},\FF\otimes\phi_u^\star\GGG)=-m\cdot\chi(\GG_{m,\bar k},\GGG)-n\cdot\chi(\GG_{m,\bar k},\FF)=
$$
\begin{equation}\label{Euler}
=m\sum_{t\in T}(n-n_t)+n\sum_{s\in S}(m-m_s).
\end{equation}
This is the generic rank of $\HHH^{-1}(\FF[1]\ast\GGG[1])$. Now let $u\in ST$, and let $R_u$ be the set of pairs $(s,t)\in S\times T$ such that $u=st$. Then $\dim\HHH^0(\FF[1]\ast\GGG[1])_u-\dim\HHH^{-1}(\FF[1]\ast\GGG[1])_u=\chi(\GG_{m,\bar k},\FF\otimes\phi_u^\star\GGG)$. Again by the Euler-Poincar\'e formula, we have
$$
\chi(\GG_{m,\bar k},\FF\otimes\phi_u^\star\GGG)=-\sum_{(s,t)\in R_u}(mn-m_sn_t)-\sum_{s\in S'_u}(mn-m_sn)-\sum_{t\in T'_u}(mn-mn_t)
$$
where $S'_u$ (resp $T'_u$) is the set of $s\in S$ such that $u/s\notin T$ (resp. the set of $t\in T$ such that $u/t\notin S$).

The Euler characteristic of $\FF[1]\ast\GGG[1]$ is then
$$
\chi(\GG_{m,\bar k},\HHH^0(\FF[1]\ast\GGG[1]))-\chi(\GG_{m,\bar k},\HHH^{-1}(\FF[1]\ast\GGG[1]))=
$$
$$
=\sum_{u\in ST}(\dim\HHH^0(\FF[1]\ast\GGG[1])_u-\dim\HHH^{-1}(\FF[1]\ast\GGG[1])_u)+
$$
$$
+\#ST\cdot\mathrm{gen.rank}(\HHH^{-1}(\FF[1]\ast\GGG[1]))+\sum_{u\in ST\cup\{0,\infty\}}\Swan_u\HHH^{-1}(\FF[1]\ast\GGG[1])=
$$
$$
=\sum_{u\in ST}\left(m\sum_{t\in T}(n-n_t)+n\sum_{s\in S}(m-m_s)-\sum_{(s,t)\in R_u}(mn-m_sn_t)\right.
$$
$$
-\left.\sum_{s\in S'_u}(mn-m_sn)-\sum_{t\in T'_u}(mn-mn_t)\right)+\sum_{u\in ST\cup\{0,\infty\}}\Swan_u\HHH^{-1}(\FF[1]\ast\GGG[1])=
$$
$$
=\sum_{u\in ST}\sum_{(s,t)\in R_u}(m-m_s)(n-n_t)+\sum_{u\in ST\cup\{0,\infty\}}\Swan_u\HHH^{-1}(\FF[1]\ast\GGG[1])=
$$
$$
=\left(\sum_{s\in S}(m-m_s)\right)\left(\sum_{t\in T}(n-n_t)\right)+\sum_{u\in ST\cup\{0,\infty\}}\Swan_u\HHH^{-1}(\FF[1]\ast\GGG[1]).
$$

Comparing with (\ref{Euler}), we conclude that $\sum_{u\in ST\cup\{0,\infty\}}\Swan_u\HHH^{-1}(\FF[1]\ast\GGG[1])=0$, that is, $\HHH^{-1}(\FF[1]\ast\GGG[1])$ is everywhere tamely ramified.
\end{proof}

\begin{cor}
 Let $\FF\in\Shg$ be an everywhere tamely ramified geometrically semisimple middle extension sheaf without negligible components. Then $\HHH^{-1}(\R(\ast\rho)\FF[1])$ is everywhere tamely ramified for every $r\geq 1$ and every representation $\rho$ of ${\mathfrak S}_r$.
\end{cor}

\begin{rem}{\rm
 Lemma \cite[Lemma 19.5]{katz2010mellin} is proved for the ``middle convolution'' (i.e. the image of the ``forget supports'' map $\FF[1]\ast_!\GGG[1]\to \FF[1]\ast_\star \GGG[1]$), while we are using ``regular'' $!$-convolution here. However, by \cite[Proposition 3.6.4]{gabber1996faisceaux}, the mapping cone of the forget supports map (and therefore the kernel of the surjective map $\FF[1]\ast_!\GGG[1]\to\FF[1]\ast_{mid}\GGG[1]$) is a succesive extension of Kummer objects $\LL_\chi[1]$ and, in particular, is smooth on $\GG_{m,\bar k}$, so the ramification points (and the non-trivial part of the inertia action at those points) of $\FF[1]\ast_!\GGG[1]$ and $\FF[1]\ast_{mid}\GGG[1]$ are the same.
}
\end{rem}

We will now try to find a good estimate for the constant $C_{\FF[1],r}$ for an everywhere tamely ramified middle extension sheaf $\FF$. 

\begin{lem}
 Let $\FF\in\Shg$ be an everywhere tamely ramified middle extension sheaf without negligible components, pure of weight $0$. For every (possibly trivial) character $\chi$ of $\bar k^\star$ and every $j\geq 1$, let $n_{\chi,j}$ be the number of Jordan blocks of size $j$ with eigenvalue $\chi$ in the local monodromies of $\FF$ at $0$ and $\infty$, $n:=\dim\HH^1_c(\GG_{m,\bar k},\FF)$ and $n_{\chi,0}:=n-\sum_{j\geq 1}n_{\chi,j}$. Let $J_\chi=\{j\geq 0|n_{\chi,j}>0\}$. Then the generic rank of $\HHH^{-1}(\Sym^{\ast r}\FF[1])$ is bounded by
\begin{equation}\label{A}
A_{\FF,r}:=\sum_\chi\sum_{(i_j)\in{\mathbb Z}_{\geq 0}^{J_\chi},\sum i_j=r}\prod_{j\in J_\chi}{{n_{\chi,j}+i_j-1}\choose{i_j}}\sum_{j\in J_\chi}ji_j.
\end{equation}
\end{lem}

\begin{proof}
Notice that the sum is actually finite, since there are only finitely many characters $\chi$ for which $J_\chi\neq\{0\}$. We will show that for any $\chi$ the sum of the sizes of the Jordan blocks in the monodromies of $\HHH^{-1}(\Sym^{\ast r}\FF[1])$ at $0$ and $\infty$ associated to the character $\chi$ is bounded above by
$$
2\cdot\sum_{(i_j)\in{\mathbb Z}_{\geq 0}^{J_\chi},\sum i_j=r}\prod_{j\in J_\chi}{{n_{\chi,j}+i_j-1}\choose{i_j}}\sum_{j\in J_\chi}ji_j.
$$

Since, by proposition \ref{tame}, $\HHH^{-1}(\Sym^{\ast r}\FF[1])$ is everywhere tamely ramified, its monodromy at $0$ and $\infty$ is a direct sum of such Jordan blocks. Therefore the sum of these quantities for all characters $\chi$ is twice the rank of $\HHH^{-1}(\Sym^{\ast r}\FF[1])$, which proves the lemma.

Fix one such $\chi$. By tensoring $\FF$ with the Kummer sheaf $\LL_{\bar\chi}$ (which does not change the hypotheses of the lemma), we can assume without loss of generality that $\chi={\mathbf 1}$ is the trivial character. Let $n_j=n_{{\mathbf 1},j}$ for $j\geq 0$. From the exact sequence of sheaves
$$
0\to j_!\FF\to j_\star\FF\to i_{0\star}\FF^{I_0}\oplus i_{\infty\star}\FF^{I_\infty}\to 0
$$
where $j:\GG_{m,k}\to\PP^1_k$, $i_0:\{0\}\to\PP^1_k$ and $i_\infty:\{\infty\}\to\PP^1_k$ are the inclusions, we get an exact sequence
$$
\HH^0(\PP^1_{\bar k},j_\star\FF)\to\FF^{I_0}\oplus\FF^{I_\infty}\to\HH^1_c(\GG_{m,\bar k},\FF)\to\HH^1(\PP^1_{\bar k},j_\star\FF)\to 0.
$$

Since $\FF$ is a middle extension pure of weight $0$, by \cite[Th\'eor\`eme 3.2.3]{deligne1980conjecture} the latter group is pure of weight $1$. On the other hand $\FF^{I_0}\oplus\FF^{I_\infty}$ is mixed of weights $\leq 0$, and in fact (cf. \cite[1.8]{deligne1980conjecture}, \cite[Theorem 7.0.7]{katz1988gauss}) the Frobenius action has $n_{j}$ eigenvalues of weight $1-j$ for every $j\geq 1$. Finally since $\FF$ has no negligible (and in particular constant) components, the first group vanishes. We conclude that $\HH^1_c(\GG_{m,\bar k},\FF)$ has $n_{j}$ Frobenius eigenvalues of weight $1-j$ for every $j\geq 0$.

Write (the semisimplification of) $\HH^1_c(\GG_{m,\bar k},\FF)$ as $\bigoplus_{j\in J_{\mathbf 1}} W_j$, where $W_j$ is pure of weight $1-j$ and $\dim W_j=n_{j}$. Then by proposition \ref{mainprop}
$$
\HH^0_c(\GG_{m,\bar k},\Sym^{\ast r}\FF[1])=\Sym^r\HH^1_c(\GG_{m,\bar k},\FF)=
$$
$$
=\bigoplus_{(i_j)\in{\mathbb Z}_{\geq 0}^{J_{\mathbf 1}},\sum i_j=r}(\bigotimes_{j\in J_{\mathbf 1}}\Sym^{i_j}W_j)
$$
where $\bigotimes_{j\in J_{\mathbf 1}}\Sym^{i_j}W_j$ is pure of weight $\sum_{j\in J_{\mathbf 1}}i_j(1-j)$ and dimension $\prod_{j\in J_{\mathbf 1}}{{n_j+i_j-1}\choose{i_j}}$.

Let $\GGG=\HHH^{-1}(\Sym^{\ast r}\FF[1])$ and consider the exact sequence
$$
0\to\HH^0(\PP^1_{\bar k},j_\star\GGG)\to\GGG^{I_0}\oplus\GGG^{I_\infty}\to\HH^1_c(\GG_{m,\bar k},\GGG)\to\HH^1(\PP^1_{\bar k},j_\star\GGG)\to 0.
$$
Since $\GGG$ is mixed of weights $\leq r-1$, every unipotent Jordan block of size $e$ in the monodromy of $\GGG$ at $0$ or $\infty$ contributes an eigenvalue of weight $\leq r-e$ to $\GGG^{I_0}\oplus\GGG^{I_\infty}$. This gives a corresponding eigenvalue of $\HH^1_c(\GG_{m,\bar k},\GGG)$ \emph{except} when it arises from something in $\HH^0(\PP^1_{\bar k},j_\star\GGG)$. But in that case, the eigenvalue appears in both monodromies at $0$ and $\infty$. So in the worst case, every two unipotent Jordan blocks of size $e$ give an eigenvalue of weight $\leq r-e$ in $\HH^1_c(\GG_{m,\bar k},\GGG)$. We conclude that the sum of the sizes of the unipotent Jordan blocks in the monodromies of $\GGG$ at $0$ and $\infty$ is bounded by $2\cdot\sum_\lambda(r-w(\lambda))$, where the sum is taken over all Frobenius eigenvalues of $\HH^1_c(\GG_{m,\bar k},\GGG)$ and $w(\lambda)$ is the weight of the eigenvalue $\lambda$.

If the trivial character does not appear in the local monodromies of $\FF$ at $0$ and $\infty$ then $\HH^1_c(\GG_{m,\bar k},\FF)$ is pure of weight $1$, and therefore $\HH^1_c(\GG_{m,\bar k},\GGG)\hookrightarrow\HH^0_c(\GG_{m,\bar k},\Sym^{\ast r}\FF[1])=\Sym^r\HH^1_c(\GG_{m,\bar k},\FF)$ is pure of weight $r-1$. So the trivial character does not appear in the local monodromies of $\GGG$ either. Otherwise, for every $(i_j)\in{\mathbb Z}_{\geq 0}^{J_{\mathbf 1}}$ such that $\sum_j i_j=r$ we get at most $\prod_{j\in J_{\mathbf 1}}{{n_j+i_j-1}\choose{i_j}}$ eigenvalues of weight $\sum_{j\in J_{\mathbf 1}}i_j(1-j)$ (since $\HH^1_c(\GG_{m,\bar k},\GGG)\hookrightarrow\HH^0_c(\GG_{m,\bar k},\Sym^{\ast r}\FF[1])=\Sym^r\HH^1_c(\GG_{m,\bar k},\FF)$). So 
$$
2\cdot\sum_\lambda(r-w(\lambda))\leq 2\cdot\sum_{(i_j)\in{\mathbb Z}_{\geq 0}^{J_{\mathbf 1}},\sum_j i_j=r}\prod_{j\in J_{\mathbf 1}}{{n_j+i_j-1}\choose{i_j}}(r-\sum_{j\in J_{\mathbf 1}}i_j(1-j))=
$$
$$
=2\cdot\sum_{(i_j)\in{\mathbb Z}_{\geq 0}^{J_{\mathbf 1}},\sum_j i_j=r}\prod_{j\in J_{\mathbf 1}}{{n_j+i_j-1}\choose{i_j}}\sum_{j\in J_{\mathbf 1}}ji_j.
$$
\end{proof}

In the same way one can prove

\begin{lem}
 With the notation and hypotheses of the previous lemma, the generic rank of $\HHH^{-1}(\wedge^{\ast r}\FF[1])$ is bounded by
\begin{equation}\label{B}
B_{\FF,r}:=\sum_\chi\sum_{(i_j)\in{\mathbb Z}_{\geq 0}^{J_\chi},\sum i_j=r}\prod_{j\in J_\chi}{{n_j}\choose{i_j}}\sum_{j\in J_\chi}ji_j.
\end{equation}
\end{lem}

\begin{prop}
 With the notation and hypotheses of the previous lemmas, the generic rank of $\HHH^{-1}(\Sym^{\ast (r-i)}\FF[1]\ast\wedge^{\ast i}\FF[1])$ is bounded by
\begin{equation}\label{C}
 M_{\FF,r,i}:= A_{\FF,r-i}{n\choose i}+B_{\FF,i}{{n-1+r-i}\choose {r-i}}.
\end{equation}
\end{prop}

\begin{proof}
Let ${\mathcal A}_j=\HHH^{-j}(\Sym^{\ast(r-i)}\FF[1])[j]$ and ${\mathcal B}_j=\HHH^{-j}(\wedge^{\ast i}\FF[1])[j]$ for $j=0,1$. Then the generic rank of $\HHH^{-1}(\Sym^{\ast (r-i)}\FF[1]\ast\wedge^{\ast i}\FF[1])$ is less than or equal to the sum of the generic ranks of $\HHH^{-1}({\mathcal A}_j\ast{\mathcal B}_k)$ for $j,k\in\{0,1\}$. For $i=j=0$ it is a punctual object, so its $\HHH^{-1}$ vanishes. For $j=1$, $k=0$, the generic rank of ${\mathcal A}_1$ gets multiplied by the dimension of the punctual object ${\mathcal B}_0$, and similarly for $j=0$, $k=1$.
 For $j=k=1$, by (the proof of) \cite[Theorem 26.1]{katz2010mellin}, there is an inequality
$$
\mathrm{gen.rank}\,\HHH^{-1}({\mathcal A}_{1}\ast{\mathcal B}_{1})\leq
$$
$$
\leq (\mathrm{gen.rank}\,{\mathcal A}_{1})\dim\HH^0_c(\GG_{m,\bar k},{\mathcal B}_{1})
+ (\mathrm{gen.rank}\,{\mathcal B}_{1})\dim\HH^0_c(\GG_{m,\bar k},{\mathcal A}_{1})
$$ 
so
$$
\mathrm{gen.rank}\,\HHH^{-1}(\Sym^{\ast (r-i)}\FF[1]\ast\wedge^{\ast i}\FF[1])\leq
$$
$$
\leq (\mathrm{gen.rank}\,\Sym^{\ast(r-i)}\FF[1])(\dim\HH^0_c(\GG_{m,\bar k},{\mathcal B}_{0})+\dim\HH^0_c(\GG_{m,\bar k},{\mathcal B}_{1}))+
$$
$$
+ (\mathrm{gen.rank}\,\wedge^{\ast i}\FF[1])(\dim\HH^0_c(\GG_{m,\bar k},{\mathcal A}_{0})+\dim\HH^0_c(\GG_{m,\bar k},{\mathcal A}_{1}))=
$$
$$
= (\mathrm{gen.rank}\,\Sym^{\ast(r-i)}\FF[1])\dim\HH^0_c(\GG_{m,\bar k},\wedge^{\ast i}\FF[1])
$$
$$
+ (\mathrm{gen.rank}\,\wedge^{\ast i}\FF[1])\dim\HH^0_c(\GG_{m,\bar k},\Sym^{\ast(r-i)}\FF[1])=
$$
$$
=(\mathrm{gen.rank}\,\Sym^{\ast(r-i)}\FF[1])\dim\wedge^i\HH^0_c(\GG_{m,\bar k},\FF[1])
$$
$$
+ (\mathrm{gen.rank}\,\wedge^{\ast i}\FF[1])\dim\Sym^{r-i}\HH^0_c(\GG_{m,\bar k},\FF[1]).
$$
so the result follows from the previous two lemmas.
\end{proof}

\begin{cor}\label{Cmult}
 Let $\FF\in\Shg$ be an everywhere tamely ramified middle extension sheaf without negligible components, pure of weight $0$. Then
$$
C_{\FF[1],r}\leq\frac{1}{2}\sum_{i=0}^r M_{\FF,r,i}=\frac{1}{2}\sum_{i=0}^r\left(A_{\FF,r-i}{n\choose i}+B_{\FF,i}{{n-1+r-i}\choose {r-i}}\right).
$$
\end{cor}

\begin{proof}
 By definition
$$
C_{\FF[1],r}=\sum_{i=0}^{r-1}\mathrm{gen.rank}\;\HHH^{-1}(\R(\ast\wedge^i\rho)\FF[1]).
$$
Using that $\Sym^{\ast (r-i)}\FF[1]\ast\wedge^{\ast i}\FF[1]=\R(\ast(\wedge^{i-1}\rho\oplus\wedge^i\rho))\FF[1]
=\R(\ast\wedge^{i-1}\rho)\FF[1]\oplus\R(\ast\wedge^i\rho)\FF[1]$ we get $$\mathrm{gen.rank}\;\HHH^{-1}(\Sym^{\ast (r-i)}\FF[1]\ast\wedge^{\ast i}\FF[1])=
$$
$$=\mathrm{gen.rank}\;\HHH^{-1}(\R(\ast\wedge^{i-1}\rho)\FF[1])+\mathrm{gen.rank}\;\HHH^{-1}(\R(\ast\wedge^i\rho)\FF[1])$$
for $i=0,\ldots,r$. Taking the sum over all $i=0,\ldots,r$ we deduce:
$$
2\cdot C_{\FF[1],r}=\sum_{i=0}^r\mathrm{gen.rank}\;\HHH^{-1}(\Sym^{\ast (r-i)}\FF[1]\ast\wedge^{\ast i}\FF[1]).
$$
We conclude by the previous proposition.
\end{proof}

\begin{exa} \label{normexa1} Let $g\in k[x]$ be a square-free polynomial of degree $d$ prime to $p$ such that $g'$ has no factors with multiplicity $\geq p$, and let $\FF\in\Shg$ be the kernel of the trace map $g_\star\QQ\to\QQ$. Then $t\in U_{\FF[1],r}(\bar k)$ for every $t$ which is not a product of $r$ critical values of $g$. In particular, for every such $t\in k_m^\star$, we have an estimate
$$
\left|\#\{x\in k_{mr}|\Nm_{k_{mr}/k_m}g(x)=t\}-\frac{q^{mr}-1}{q^m-1}\right|\leq C_{\FF[1],r} q^{\frac{m(r-1)}{2}}.
$$
Morover, we have a bound
$$C_{\FF[1],r}\leq\frac{1}{2}\sum_{i=0}^r{{d-1}\choose i}\left[(r+i){{d-2+r-i}\choose{r-i}}+(d-1)\sum_{j=0}^{r-i-1}{{d-3+j}\choose{j}}(r-i-j)\right].$$
\end{exa}

\begin{rem}\label{compare}\emph{The given bound for $C_{\FF[1],r}$ is polynomial in $r$ if $d$ is fixed. Compare with \cite[Theorem 3.2]{rl2010number}, where the (exponential in $r$) bound $r(d-1)^r$ was obtained using Weil descent.}
 
\end{rem}

\begin{proof}
The first part is a consequence of propositions \ref{negligible} and \ref{Sr}, since $\FF$ is tamely ramified everywhere and has no negligible components (because $\FF$ is a middle extension sheaf whose monodromy is trivial at $0$ and splits as the direct sum of all non-trivial characters with trivial $d$-th power at infinity). The estimate follows from the equality $f_{g_\star\QQ}^{\Nm,r}=f_{\QQ}^{\Nm,r}+f_{\FF}^{\Nm,r}$.

From the exact sequence
$$
0\to\FF\to g_\star\QQ\to\QQ\to 0
$$
we deduce $\dim\HH^1_c(\GG_{m,\bar k},\FF)=\dim\HH^1_c(\AAA^1_{\bar k},\FF)+\dim(\FF_0)=\dim\HH^1_c(\AAA^1_{\bar k},g_\star\QQ)-\dim\HH^1_c(\AAA^1_{\bar k})+\dim(\FF_0)=\dim(\FF_0)=d-1$.

If $\chi={\mathbf 1}$ is the trivial character, $n_{\chi,1}=d-1$ and $n_{\chi,j}=0$ for every $j\neq 1$. If $\chi\neq{\mathbf 1}$ but $\chi^d={\mathbf 1}$, $n_{\chi,1}=1$, $n_{\chi,0}=d-2$ and $n_{\chi,j}=0$ for every $j>1$. By equations (\ref{A}) and (\ref{B}) we get, for every $i=0,\ldots,r$:
$$
A_{\FF,r-i}={{d-2+r-i}\choose{r-i}}(r-i)+(d-1)\sum_{j=0}^{r-i-1}{{d-3+j}\choose{j}}(r-i-j)
$$
and
$$
B_{\FF,i}= {{d-1}\choose i}\cdot i+(d-1){{d-2}\choose{i-1}}=2i{{d-1}\choose i}
$$
so
$$
M_{\FF,r,i}={{d-1}\choose i}\left[(r+i){{d-2+r-i}\choose{r-i}}+(d-1)\sum_{j=0}^{r-i-1}{{d-3+j}\choose{j}}(r-i-j)\right].
$$
We conclude by corollary \ref{Cmult}.
\end{proof}

\begin{rem}\label{remark}\emph{If $g$ is square-free, it follows from \cite[Lemma 3.1]{rl2010number} that $\HHH^0(\FF[1]^{\ast r})=0$. In particular, $U_{\FF[1],r}=\GG_{m,\bar k}$ for every $r\geq 1$, so the estimate holds for every $m\geq 1$ and every $t\in k_m^\star$.}
\end{rem}

\begin{exa} \label{normexa2} Let $g\in k[x]$ be a polynomial of degree $d$ prime to $p$ and $\chi:k^\star\to\QQ^\star$ a non-trivial multiplicative character of order $n$. Suppose that $g$ is not a power of $x$, and no root of $g$ has multiplicity divisible by $n$. Let $\FF=\LL_{\chi(g)}$. Then $t\in U_{\FF[1],r}(\bar k)$ for every $t$ which is not a product of $r$ roots of $g$. In particular, for every such $t\in k_m^\star$, we have an estimate
$$
\left|\sum_{\Nm_{k_{mr}/k_m}(x)=t}\chi(\Nm_{k_{mr}/k}(g(x)))\right|\leq C_{\FF[1],r} q^{\frac{m(r-1)}{2}}.
$$
Let $e$ be the largest power of $x$ that divides $g(x)$. Then if $\chi^e\neq\chi^d$, we have a bound
$$
C_{\FF[1],r}\leq \frac{1}{2}\sum_{i=0}^r2\left({{a-1+r-i}\choose{r-i}}{{a-1}\choose{i-1}}+{a\choose i}\sum_{j=0}^{r-i-1}{{a+j-2}\choose j}(r-i-j)\right)
$$
and, if $\chi^e=\chi^d$,
$$
C_{\FF[1],r}\leq \frac{1}{2}\sum_{i=0}^r\left(2{{a-1+r-i}\choose{r-i}}{{a-1}\choose{i-1}}+{a\choose i}\sum_{j=0}^{r-i-1}{{a+j-3}\choose j}(r+1-i-j)(r-i-j)\right)
$$
where $a$ is the number of distinct roots of $g$ in $\bar k^\star$.
\end{exa}

\begin{proof}
 Again this is just applying proposition \ref{Sr} and formulas (\ref{A}), (\ref{B}) and (\ref{C}) for the rank. The hypotheses on $g$ imply that $\LL_{\chi(g)}$ is a middle extension and ramified at least at one point of $\GG_{m,\bar k}$, and in particular is not negligible. The dimension of $\HH^1_c(\GG_{m,\bar k},\LL_{\chi(g)})$ is $a$ by the Euler-Poincar\'e formula, since $\LL_{\chi(g)}$ is everywhere tamely ramified. The monodromy at $0$ is the character $\chi^e$, and the monodromy at infinity is $\chi^d$, hence the different bounds for $C_{\FF[1],r}$ depending on them being equal or not.
\end{proof}

In order to obtain sharper results we will make use of a certain algebraic group, the equivalent to the monodromy group of the Fourier transform in the trace case. This group is defined and studied in \cite{katz2010mellin}. Given a geometrically semisimple (e.g. pure of some weight $w$) object $\PPP\in\Perv$ without negligible components, let $\langle\PPP\rangle$ be the full subcategory of the Tannakian category of perverse sheaves on $\GG_{m,\bar k}$ modulo negligible sheaves with the convolution operator \cite[Th\'eor\`eme 3.7.5]{gabber1996faisceaux} tensor-generated by $\PPP\otimes\bar k$. By the fundamental theorem of Tannakian categories \cite[Theorem 2.11]{deligne1981tannakian}, $\langle\PPP\rangle$ is tensor-equivalent to the category of representations of a reductive algebraic group $G\subseteq\mathrm{GL}(V)$, where $V=\HH^0(\AAA^1_{\bar k},j_{0!}\PPP)$ is the fibre functor evaluated at $\PPP$ \cite[Theorem 3.1]{katz2010mellin}. Under this equivalence, the class of $\PPP$ corresponds to the ``standard'' representation $G\hookrightarrow\mathrm{GL}(V)$, and the class of $\PPP^{\ast r}$ (respectively $\Sym^{\ast r}\PPP$, $\wedge^{\ast r}\PPP$) corresponds to its $r$-th tensor power (resp. its $r$-th symmetric power, its $r$-th alternating power). More generally, for every finite dimensional representation $\rho:{\mathfrak S}_r\to\GL(W)$ of ${\mathfrak S}_r$, the class of $\R(\ast\rho)\PPP$ corresponds to $\Hom_{{\mathfrak S}_r}(W,V^{\otimes r})$.

\begin{prop}\label{subcharacters}
 Let $\PPP\in\Perv$ be a perverse sheaf pure of weight $w\in {\mathbb Z}$ without negligible components, $G\subseteq\mathrm{GL}(V)$ the corresponding reductive algebraic group and $r\geq 1$ an integer. Then

\begin{enumerate}
 \item  $1\in U_{\PPP,r}(\bar k)$ if and only the fixed subspace of the representation $\Sym^{r-i}V\otimes\wedge^iV$ of $G$ is zero for every $i=0,\ldots,r$.
 \item Given a prime to $p$ integer $n$, the $n$-th roots of unity are in $U_{\PPP,r}(\bar k)$ if and only if the representation $\Sym^{r-i}V\otimes\wedge^iV$ of $G$ does not contain a subcharacter with trivial $n$-th power for any $i=0,\ldots,r$.
  \item $U_{\PPP,r}=\GG_{m,k}$ if and only if the representation $\Sym^{r-i}V\otimes\wedge^iV$ of $G$ does not contain subcharacters of finite prime to $p$ order for any $i=0,\ldots,r$.
\end{enumerate}
\end{prop}

\begin{proof}
 (1) and (3) are direct consequences of (2). Let $n$ be a prime to $p$ positive integer, and let $\zeta\in \bar k^\star -U_{\PPP,r}(\bar k)$ such that $\zeta^n=1$. Then by definition of $U_{\PPP,r}$ there is some $i=0,\ldots,r-1$ such that $\zeta$ is in the support of $\HHH^0(\R(\ast\wedge^{i}\rho)\PPP)$. Then $\zeta$ is in the support of $\HHH^0(\Sym^{\ast(r-i)}\PPP\ast\wedge^{\ast i}\PPP)=\HHH^0(\R(\ast\wedge^{i-1}\rho)\PPP\oplus\R(\ast\wedge^i\rho)\PPP)$ (cf. the proof of proposition \ref{young}). In other words, $\Sym^{\ast(r-i)}\PPP\ast\wedge^{\ast i}\PPP$ contains the punctual object $\delta_\zeta$ as an irreducible component. Regarding their classes in the Tannakian category $\Perv/{\mathcal N}eg$ as representations of $G$, this means that $\Sym^{r-i}V\otimes\wedge^{ i}V$ contains the irreducible subrepresentation associated to the object $\delta_\zeta$, which is a character with trivial $n$-th power (since $\delta_\zeta^{\ast n}=\delta_{\zeta^n}=\delta_1$ is the identity object).

Conversely, every subrepresentation of $\Sym^{r-i}V\otimes\wedge^{ i}V$ which is a character of order divisible by $n$ gives an irreducible component of $\Sym^{\ast(r-i)}\PPP\ast\wedge^{\ast i}\PPP$ of the form $\delta_\zeta$ for some $\zeta\in\bar k^\star$ \cite[Theorem 6.4]{katz2010mellin}, and $\zeta^n=1$. So, in that case, $\zeta$ is in the support of $\HHH^0(\Sym^{\ast(r-i)}\PPP\ast\wedge^{\ast i}\PPP)=\HHH^0(\R(\ast\wedge^{i-1}\rho)\PPP\oplus\R(\ast\wedge^i\rho)\PPP)$ for some $i$, so it must be in the support of $\HHH^0(\R(\ast\wedge^i\rho)\PPP)$ for some $i$ and therefore is not in $U_{\PPP,r}(\bar k)$.
\end{proof}

\begin{cor}
 Let $n$ be the order of $G/G_0$, where $G_0$ is the identity connected component of $G$. Then $\bar k^\star-\mu_n(\bar k)\subseteq U_{\PPP,r}(\bar k)$ for every $r\geq 1$, where $\mu_n(\bar k):=\{x\in\bar k|x^n=1\}$. In particular, if $G$ is connected, $\bar k^\star-\{1\}\subseteq U_{\PPP,r}(\bar k)$ for every $r\geq 1$.
\end{cor}

\begin{proof}
Let $z\in\bar k^\star$, and suppose that $z\notin U_{\PPP,r}(\bar k)$. Then by the proof of the previous proposition, $\Sym^{r-i}V\otimes\wedge^{ i}V$ contains a subcharacter of order $a$ for some $i=0,\ldots,r$, where $a$ is the multiplicative order of $z$. Let $G'$ be the kernel of that subcharacter, then $G'$ is a closed normal subgroup of $G$, and $G/G'$ is a quotient of $G/G_0$ of order $a$. Therefore $a=|G/G'|$ divides $n=|G/G_0|$, so $z$ is an $n$-th root of unity.
\end{proof}

\begin{rem}
  \emph{By \cite[Theorem 6.5]{katz2010mellin}, under these hypotheses $G/G_0$ is actually cyclic of prime-to-$p$ order $n$, and this group is isomorphic to the group of $\zeta\in\bar k^\star$ such that $\delta_z$ is in the Tannakian subcategory of $\Perv/{\mathcal N}eg$ tensor-generated by $\PPP\otimes\bar k$.}
\end{rem}

\begin{cor}\label{constants}
 If $G$ contains the scalars $\QQ^\star$, then $U_{\PPP,r}=\GG_{m,k}$ for every $r\geq 1$.
\end{cor}

\begin{proof}
 By proposition \ref{subcharacters}(3), it suffices to show that the representation $V^{\otimes r}$ of $G$ does not have subcharacters of finite order. A scalar $\lambda\in\QQ^\star\subseteq G$ acts on $V^{\otimes r}$ by multiplication by $\lambda^r$. In particular, on any $G$-invariant subspace of $V^{\otimes r}$ the action of the quotient $\QQ^\star/{\mathbb \mu}_r$ is faithful, so it can never factor through a finite quotient. 
\end{proof}

\begin{exa}
 Let $g\in k[x]$ be a Morse polynomial of degree $d$ prime to $p$ such that its set of critical values is not isomorphic to a multiplicative translate of itself. Let $\FF$ be the kernel of the trace map $g_\star\QQ\to\QQ$. Then $U_{\FF[1],r}=\GG_{m,k}$ for every $r\geq 1$ and for every $r\geq 1$ and every $t\in k_m^\star$ we have
$$
\left|\#\{x\in k_{mr}|\Nm_{k_{mr}/k_m}g(x)=t\}-\frac{q^{mr}-1}{q^m-1}\right|\leq C_{\FF[1],r} q^{\frac{m(r-1)}{2}}
$$
where $C_{\FF[1],r}$ is bounded as in \ref{normexa1} if $g$ is square-free. In particular, for every $e$ dividing $q^m-1$ we have the estimate
 $$
\left|\#\{(x,y)\in k_{mr}^2|y^{(q^m-1)/e}=g(x)\}-q^{mr}-(\delta-1)\right|
\leq C_{\FF[1],r} (q^m-1)q^{\frac{m(r-1)}{2}}
$$
where $\delta$ is the number of roots of $g$ in $k_m$.
\end{exa}

\begin{proof}
 By \cite[Theorem 17.6]{katz2010mellin}, under these hypotheses $G$ is the entire $\mathrm{GL}(V)$, which contains the scalars, so the result follows from corollary \ref{constants}.

The second estimate is an easy consequence of the identity
$$
\#\{(x,y)\in k_{mr}^2|y^{(q^m-1)/e}=g(x)\}=\delta+\frac{q^m-1}{e}\sum_{\lambda^e=1}\#\{x\in k_{mr}|\Nm_{k_{mr}/k_m}(g(x))=\lambda\}.
$$
\end{proof}

\begin{rem}
\emph{As noted above (remark \ref{remark}), if $g$ is square-free of degree prime to $p$ the result holds without any further hypotheses on $g$.}
\end{rem}

\begin{exa}
 Let $g\in k[x]$ be a square-free polynomial of degree $d$ with $g(0)\neq 0$ which is not of the form $h(x^n)$ for any $n\geq 2$, let $\chi:k^\star\to\QQ^\star$ be a multiplicative character such that $\chi^d$ is non-trivial and $\FF=\LL_{\chi(g)}$. Then $U_{\FF[1],r}=\GG_{m,k}$ for every $r\geq 1$. In particular, for every $r\geq 1$ and every $t\in k_m^\star$ one has
$$
\left|\sum_{\Nm_{k_{mr}/k_m}(x)=t}\chi(\Nm_{k_{mr}/k}(g(x)))\right|\leq C_{\FF[1],r} q^{\frac{m(r-1)}{2}}
$$
where $C_{\FF[1],r}$ is bounded as in \ref{normexa2}.
\end{exa}

\begin{proof}
 By \cite[Theorem 17.5]{katz2010mellin}, under these hypotheses $G=\mathrm{GL}(V)$ contains the scalars, so the result follows from corollary \ref{constants}.
 \end{proof}

\begin{exa}
 Let $g\in k[x]=\sum a_ix^i$ be a square-free polynomial of degree $d$ with $a_0\neq 0$ which is not of the form $h(x^n)$ for any $n\geq 2$ and such that $x^dg(1/x)$ is not a scalar multiple of a multiplicative translate of $g$, let $\chi:k^\star\to\QQ^\star$ be a multiplicative character such that $\chi^d$ is trivial and $\PPP=\LL_{\chi(g)}[1]$. Then $U_{\PPP,r}=\GG_{m,k}$ for every $r\neq d$ and $U_{\PPP,d}=\GG_{m,k}-\{(-1)^da_0\}$. In particular, for every $r\geq 1$ and every $t\in k_m^\star$ (except $t=(-1)^da_0$ when $r=d$) one has
$$
\left|\sum_{\Nm_{k_{mr}/k_m}(x)=t}\chi(\Nm_{k_{mr}/k}(g(x)))\right|\leq C_{\PPP,r} q^{\frac{m(r-1)}{2}}
$$
where $C_{\PPP,r}$ is bounded as in \ref{normexa2}.
\end{exa}

\begin{proof}
 By \cite[Theorem 23.1]{katz2010mellin}, under these hypotheses $G$ contains $\mathrm{SL}(V)$. Since $\mathrm{SL}(V)$ is connected, $\Sym^{r-i}V\otimes\wedge^i V=\Hom(\wedge^{d-i}V,\Sym^{r-i}V)$ has a subcharacter of finite order as a $\mathrm{SL}(V)$-representation if and only if it has non-trivial $\mathrm{SL}(V)$-fixed subspace, which happens only for $d-i=r-i=0$ or $1$. So for $r\neq d$ it has no subcharacters of finite order and therefore $U_{\PPP,r}=\GG_{m,k}$ by proposition \ref{subcharacters}, (3).

For $r=d$ and $i=r,r-1$, $\Sym^{r-i}V\otimes\wedge^i V$ has $\mathrm{SL}(V)$-fixed subspace of dimension $1$, on which $G$ acts via the determinant \cite[Corollary 4.2]{rlwan2010}. So $\Sym^{r-i}V\otimes\wedge^i V$ contains the determinant character as its only subcharacter of finite order, which by \cite[Theorem 23.1]{katz2010mellin} is the punctual sheaf $\delta_{(-1)^da_0}[0]$ as an element of $\Perv/{\mathcal N}eg$. In other words, the only possible punctual geometric irreducible component of $\Sym^{\ast(d-i)}\PPP\ast\wedge^{\ast i}\PPP$ is $\delta_{(-1)^da_0}[0]$. Therefore $U_{\PPP,d}=\GG_{m,k}-\{(-1)^da_0\}$.
 \end{proof}

\begin{exa}
 Let $g\in k[x]$ be an Artin-Schreier-reduced polynomial (i.e. it has no monomials with divisible by $p$ exponent) of degree $d$ prime to $p$ which is not of the form $h(x^n)$ for any $n\geq 2$, let $\psi:k\to\QQ^\star$ be a non-trivial additive character and $\PPP$ the shifted Artin-Schreier sheaf $\LL_{\psi(g)}[1]$. Then $U_{\PPP,r}=\GG_{m,k}$ for every $r\geq 1$. In particular, for every $r\geq 1$ and every $t\in k_m^\star$ one has
$$
\left|\sum_{\Nm_{k_{mr}/k_m}(x)=t}\psi(\Trace_{k_{mr}/k}(g(x)))\right|\leq C_{\FF,r} q^{\frac{m(r-1)}{2}}.
$$
\end{exa}

\begin{proof}
 By \cite[Theorem 17.4]{katz2010mellin}, under these hypotheses $G=\mathrm{GL}(V)$ contains the scalars, so the result follows from corollary \ref{constants}.
 \end{proof}

\begin{rem}
 \emph{By \cite[Theorem 5.1]{katz1988gauss}, in the previous example $\PPP^{\ast r}$, and a fortiori its subobjects $\R(\ast\wedge^i\rho)\PPP$, are smooth on $\GG_{m,k}$, tamely ramified at $0$ and totally wild at $\infty$. If conjecture \cite[7.6]{katz1988gauss} is true, then $\PPP^{\ast r}$ has a single slope $d/r$ at infinity. In that case, 
$$
\mathrm{gen.rank}(\R(\ast\wedge^i\rho)\PPP)=\frac{r}{d}\Swan_\infty(\R(\ast\wedge^i\rho)\PPP)=
$$
$$
=\frac{r}{d}\dim\HH^0_c(\GG_{m,\bar k},\R(\ast\wedge^i\rho)\PPP)=
$$
$$
=\frac{r}{d}\dim\R(\wedge^i\rho)\HH^0_c(\GG_{m,\bar k},\PPP)=\frac{r}{d}{{d+r-i-1}\choose {r}}{r-1\choose i}
$$
by remark \ref{rank}, so we would obtain a bound
$$
C_{\PPP,r}\leq\frac{r}{d}\sum_{i=0}^{r-1}{{d+r-i-1}\choose {r}}{r-1\choose i}.
$$
Without the conjecture, we can only assure that
$$
C_{\PPP,r}\leq s^{-1}\sum_{i=0}^{r-1}{{d+r-i-1}\choose {r}}{r-1\choose i}.
$$
where $s$ is the smallest slope of $\PPP^{\ast r}$ at infinity.
}
\end{rem}

\section{The situation over $k_r$}

In this last section we will briefly discuss what happens when the sheaf or derived category object $\FF$ is defined over the larger field $k_r$, but not over $k$. 

 Let $\FF\in K_0(G\otimes{k_r},\QQ)$ be an object defined on the geometrically connected commutative group scheme $G\otimes{k_r}$ over $k_r$. In principle it makes sense to consider its $r$-th local norm $L$-function $L^{\Nm,r}(\FF,k,t;T)$ at any $t\in k$, since $\Tr(\Frob_{k_{rs},u}|\FF_{\bar u})$ is defined for every $s\geq 1$ and every $u\in G(k_{rs})$. So we can ask, is this $L$-function rational? The answer is negative in general, as we can see in the following example for $G=\AAA^1$:

 Let $q$ be odd, $a\in k_2-k$ such that $\Tr_{k_2/k}(a)=0$ and $\FF=i_{a\star}\QQ$ the skyscraper sheaf supported on $a$. Then 
$$\sum_{\Tr_{k_{2s}/k_s}(u)=0} \Tr(\Frob_{k_{2s},u}|\FF_{\bar u})=\left\{\begin{array}{ll}
                                                                          1 & \mbox{if $\Tr_{k_{2s}/k_s}(a)=0$} \\
0 & \mbox{otherwise}
                                                                         \end{array}\right.
$$
Since $a^q+a=0$, by induction for every $s$ we have $a^{q^s}=(-1)^sa$, so $\Tr_{k_{2s}/k_s}(a)=0$ if $s$ is odd and $2a\neq 0$ if $s$ is even. Therefore
$$
L^{\Tr,2}(\FF,k,0;T)=\exp\sum_{s=1}^\infty \frac{T^{2s-1}}{2s-1}=\left(\frac{1+T}{1-T}\right)^{1/2}
$$
is not rational.

However, we have the following slightly weaker but equally useful result:
\begin{prop}
 Let $\FF\in K_0(G\otimes k_r,\QQ)$. Then $L^{\Nm,r}(\FF,k,t;T)^r$ is rational for every $t\in G(k)$, and all its reciprocal roots and poles have integral $q$-weight.
\end{prop}

\begin{proof}
 By additivity, we may assume that $\FF$ is a single sheaf. Let $\GGG$ be the direct sum of all $\sigma^\star\FF$ for every $\sigma\in\Gal(k_r/k)$. Then $\GGG$ is invariant under $\Gal(k_r/k)$, so by descent there is a sheaf $\GGG_0\in K_0(G,\QQ)$ such that $\GGG=\pi_r^\star \GGG_0$, where $\pi_r:G\otimes k_r\to G$ is the projection. Therefore $L^{\Nm,r}(\GGG,k,t;T)$ is rational for every $t\in k$ by corollary \ref{maincor}.

Again by additivity, we have
$$
L^{\Nm,r}(\GGG,k,t;T)=\prod_\sigma L^{\Nm,r}(\sigma^\star \FF,k,t;T)
$$
where the product is taken over all $\sigma\in\Gal(k_r/k)$. We will conclude by showing that for every $\sigma\in\Gal(k_r/k)$ and every $t\in G(k)$ there is an equality $L^{\Nm,r}(\sigma^\star \FF,k,t;T)=L^{\Nm,r}(\FF,k,t;T)$.

Indeed we have, by definition,
$$
L^{\Nm,r}(\sigma^\star \FF,k,t;T)=\exp\sum_{s\geq 1}f_{\sigma^\star \FF}^{\Nm,r}(k_s,t)\frac{T^s}{s}
$$
and, for every $s\geq 1$,
$$
f_{\sigma^\star \FF}^{\Nm,r}(k_s,t)=\sum_{\Nm_{k_{sr}/k_s}(u)=t}f_{\sigma^\star \FF}(k_{sr},u)=\sum_{\Nm_{k_{sr}/k_s}(u)=t}f_{\FF}(k_{sr},\sigma(u))=
$$
$$
=\sum_{\Nm_{k_{sr}/k_s}(u)=t}f_{\FF}(k_{sr},u)=f_\FF^{\Nm,r}(k_s,t)
$$
since $u\mapsto \sigma(u)$ is a permutation of the set of $u\in G(k_{sr})$ such that $\Nm_{k_{sr}/k_s}(u)=t$.
\end{proof}

 In particular, many results proved in the previous sections can be applied in this case for every $t\in G(k)$ via the equality $f_\GGG^{\Nm,r}(k,t)=r\cdot f_\FF^{\Nm,r}(k,t)$. Most notably we have:

\begin{cor}
 Let $G$ be a geometrically connected affine commutative algebraic group over $k$ of dimension $1$. Let $\PPP\in\Perv(G\otimes k_r,\QQ)$ be pure of weight $1$ without negligible components, and let ${\mathcal Q}$ be the direct sum of $\sigma^\star\PPP$ for $\sigma\in\Gal(k_r/k)$ (and also its descent to $G_k$). Then for every $t\in U_{{\mathcal Q},r}(k)$ we have the estimate
$$
|f^{\Nm,r}_\PPP(k,t)|\leq \frac{C_{{\mathcal Q},r}}{r}q^{\frac{r-1}{2}}.
$$ 
\end{cor}
which is a direct consequence of corollary \ref{UC} and the previous proposition.

\bibliographystyle{amsplain}
\bibliography{bibliography}

\providecommand{\bysame}{\leavevmode\hbox to3em{\hrulefill}\thinspace}
\providecommand{\MR}{\relax\ifhmode\unskip\space\fi MR }
\providecommand{\MRhref}[2]{%
  \href{http://www.ams.org/mathscinet-getitem?mr=#1}{#2}
}
\providecommand{\href}[2]{#2}
\begin{thebibliography}{10}

\bibitem{beilinson1982faisceaux}
A.~Beilinson, J.~Bernstein, and P.~Deligne, \emph{Faisceaux pervers, analyse et
  topologie sur les espaces singuliers ({I})}, Ast{\'e}risque \textbf{100}
  (1982).

\bibitem{brylinski1986transformations}
J.L. Brylinski, \emph{Transformations canoniques, dualit{\'e} projective,
  th{\'e}orie de {L}efschetz, transformations de {F}ourier et sommes
  trigonom{\'e}triques}, Ast{\'e}risque \textbf{140} (1986), 3--134.

\bibitem{chai2004character}
C.L. Chai and W.C.W. Li, \emph{Character sums, automorphic forms,
  equidistribution, and {R}amanujan graphs. part {II}. eigenvalues of
  {T}erras}, Forum Mathematicum \textbf{16} (2004), no.~5, 631--661.

\bibitem{deligne569application}
P.~Deligne, \emph{Application de la formule des traces aux sommes
  trigonom{\'e}triques, dans {C}ohomologie {E}tale, {S}{\'e}minaire de
  {G}{\'e}om{\'e}trie {A}lg{\'e}brique du {B}ois-{M}arie, {SGA} 4 1/2}, Lecture
  Notes in Math \textbf{569}, 168--232.

\bibitem{deligne569fonction}
\bysame, \emph{Fonction {L} mod $\ell^n$, dans {C}ohomologie {E}tale,
  {S}{\'e}minaire de {G}{\'e}om{\'e}trie {A}lg{\'e}brique du {B}ois-{M}arie,
  {SGA} 4 1/2}, Lecture Notes in Math \textbf{569}, 110--128.

\bibitem{deligne1974conjecture}
\bysame, \emph{La conjecture de {W}eil. {I}}, Publications Math{\'e}matiques de
  l'IH{\'E}S \textbf{43} (1974), no.~1, 273--307.

\bibitem{deligne1980conjecture}
\bysame, \emph{La conjecture de {W}eil. {II}}, Publications Math{\'e}matiques
  de l'IH{\'E}S \textbf{52} (1980), no.~1, 137--252.

\bibitem{deligne1981tannakian}
P.~Deligne and J.~Milne, \emph{Tannakian categories}, Hodge cycles, motives,
  and Shimura varieties (1981), 101--228.

\bibitem{fu2004moment}
L.~Fu and D.~Wan, \emph{Moment {L}-functions, partial {L}-functions and partial
  exponential sums}, Mathematische Annalen \textbf{328} (2004), no.~1,
  193--228.

\bibitem{fulton1996representation}
W.~Fulton and J.~Harris, \emph{Representation theory: A first course}, Graduate
  Texts in Mathematics, vol. 129, Springer, 2004.

\bibitem{gabber1996faisceaux}
O.~Gabber and F.~Loeser, \emph{Faisceaux pervers {l}-adiques sur un tore}, Duke
  Mathematical Journal \textbf{83} (1996), no.~3, 501--606.

\bibitem{grothendieck1977cohomologie}
A.~Grothendieck, \emph{Cohomologie {l}-adique et fonctions {L} ({SGA V})},
  Lecture Notes in Mathematics, vol. 589, Springer-Verlag, 1977.

\bibitem{katz-monodromy}
N.M. Katz, \emph{On the monodromy groups attached to certain families of
  exponential sums}, Duke Mathematical Journal \textbf{54} (1987).

\bibitem{katz1988gauss}
\bysame, \emph{{G}auss sums, {K}loosterman sums, and monodromy groups}, Annals
  of Mathematics Studies, vol. 116, Princeton University Press, 1988.

\bibitem{katz1990esa}
\bysame, \emph{Exponential sums and differential equations}, Annals of
  Mathematics Studies, vol. 124, Princeton University Press, 1990.

\bibitem{katz1993estimates}
\bysame, \emph{Estimates for {S}oto-{A}ndrade sums}, J. reine angew. Math
  \textbf{438} (1993), 143--161.

\bibitem{katz1995note}
\bysame, \emph{A note on exponential sums}, Finite Fields and Their
  Applications \textbf{1} (1995), no.~3, 395--398.

\bibitem{katz1996rls}
\bysame, \emph{Rigid local systems}, Annals of Mathematics Studies, vol. 139,
  Princeton University Press, 1996.

\bibitem{katz2001frobenius}
\bysame, \emph{{F}robenius-{S}chur indicator and the ubiquity of
  {B}rock-{G}ranville quadratic excess}, Finite Fields and Their Applications
  \textbf{7} (2001), no.~1, 45--69.

\bibitem{katz2010mellin}
\bysame, \emph{{S}ato-{T}ate theorems for finite-field {M}ellin transforms},
  preprint, available at \url{http://math.princeton.edu/~nmk} (2010).

\bibitem{katz1985transformation}
N.M. Katz and G.~Laumon, \emph{Transformation de {F}ourier et majoration de
  sommes exponentielles}, Publications Math{\'e}matiques de l'IH{\'E}S
  \textbf{62} (1985), no.~1, 145--202.

\bibitem{laumon1987transformation}
G.~Laumon, \emph{Transformation de {F}ourier, constantes d'{\'e}quations
  fonctionnelles et conjecture de {W}eil}, Publications Math{\'e}matiques de
  l'IH{\'E}S \textbf{65} (1987), no.~1, 131--210.

\bibitem{li2006character}
W.C.W. Li, \emph{Character sums over norm groups}, Finite Fields and Their
  Applications \textbf{12} (2006), no.~1, 1--15.

\bibitem{rl2010number}
A.~Rojas-Le{\'o}n, \emph{On the number of rational points on curves over finite
  fields with many automorphisms}, preprint arXiv:1005.4078v2 (2010).

\bibitem{rlwan2010}
A.~Rojas-Le{\'o}n and D.~Wan, \emph{Improvements of the {W}eil bound for
  {A}rtin-{S}chreier curves}, Mathematische Annalen (online, DOI
  10.1007/s00208-010-0606-3) (2010).

\end{thebibliography}

\end{document}